\documentclass[11pt]{amsart}

\usepackage{amssymb}
\usepackage{url}

\usepackage{wasysym}
\usepackage{amsmath}
\usepackage{tikz-cd}

\usepackage{mathrsfs,multicol,float}
\usepackage[pagebackref, hypertexnames=false, colorlinks, citecolor=red, linkcolor=red]{hyperref}

	\normalbaselines						  %
\def\RR{\mathbb{R}}
\def\ZZ{\mathbb{Z}}
\def\CC{\mathbb{C}}
\def\NN{\mathbb{N}}

\newcommand{\al}{{\alpha}}
\newcommand{\la}{{\lambda}}
\newcommand{\f}{{\varphi}}

\newcommand{\bs}{{\boldsymbol{s}}}
\newcommand{\bal}{{\boldsymbol{\alpha}}}

\newcommand{\wt}{\widetilde}

\newcommand{\Wr}{\text{Wr}}

\makeatletter
\def\Ddots{\mathinner{\mkern1mu\raise\p@
\vbox{\kern7\p@\hbox{.}}\mkern2mu
\raise4\p@\hbox{.}\mkern2mu\raise7\p@\hbox{.}\mkern1mu}}
\makeatother

\newcommand{\cH}{\mathcal{H}}

\newcommand{\fC}{\mathfrak{C}}
\newcommand{\fD}{\mathfrak{D}}

\newcommand{\cA}{\mathcal{A}}

\newcommand{\cW}{\mathcal{W}}

\newcommand{\cD}{\mathcal{D}}

\newcommand{\cL}{\mathcal{L}}

\newcommand{\cQ}{\mathcal{Q}}
\newcommand{\cR}{\mathcal{R}}

\DeclareMathOperator{\Ker}{Ker}

\DeclareMathOperator{\dom}{dom}

\newcommand{\ci}[1]{_{ {}_{\scriptstyle #1}}}
\newcommand{\ti}[1]{_{\scriptstyle \text{\rm #1}}}

\catcode`@=11
\chardef\mathlig@atcode\count255

\def\actively#1#2{\begingroup\uccode`\~=`#2\relax\uppercase{\endgroup#1~}}
\def\mathlig@gobble{\afterassignment\mathlig@next@cmd\let\mathlig@next= }
\def\mathlig@delim{\mathlig@delim}
\def\mathlig@defcs#1{\expandafter\def\csname#1\endcsname}
\def\mathlig@let@cs#1#2{\expandafter\let\expandafter#1\csname#2\endcsname}
\def\mathlig@appendcs#1#2{\expandafter\edef\csname#1\endcsname{\csname#1\endcsname#2}}

\def\mathlig#1#2{\mathlig@checklig#1\mathlig@end\mathlig@defcs{mathlig@back@#1}{#2}\ignorespaces}

\def\mathlig@checklig#1#2\mathlig@end{%
 \expandafter\ifx\csname mathlig@forw@#1\endcsname\relax
 \expandafter\mathchardef\csname mathlig@back@#1\endcsname=\mathcode`#1%
 \mathcode`#1"8000\actively\def#1{\csname mathlig@look@#1\endcsname}%
 \mathlig@dolig#1\mathlig@delim
\fi
\mathlig@checksuffix#1#2\mathlig@end
}

\def\mathlig@checksuffix#1#2\mathlig@end{%
\ifx\mathlig@delim#2\mathlig@delim\relax\else\mathlig@checksuffix@{#1}#2\mathlig@end\fi
}
\def\mathlig@checksuffix@#1#2#3\mathlig@end{%
\expandafter\ifx\csname mathlig@forw@#1#2\endcsname\relax\mathlig@dosuffix{#1}{#2}\fi
\mathlig@checksuffix{#1#2}#3\mathlig@end
}

\def\mathlig@dosuffix#1#2{%
\mathlig@appendcs{mathlig@toks@#1}{#2}%
\mathlig@dolig{#1}{#2}\mathlig@delim
}

\def\mathlig@dolig#1#2\mathlig@delim{%
 \mathlig@defcs{mathlig@look@#1#2}{%
 \mathlig@let@cs\mathlig@next{mathlig@forw@#1#2}\futurelet\mathlig@next@tok\mathlig@next}%
 \mathlig@defcs{mathlig@forw@#1#2}{%
  \mathlig@let@cs\mathlig@next{mathlig@back@#1#2}%
  \mathlig@let@cs\checker{mathlig@chck@#1#2}%
  \mathlig@let@cs\mathligtoks{mathlig@toks@#1#2}%
  \expandafter\ifx\expandafter\mathlig@delim\mathligtoks\mathlig@delim\relax\else
  \expandafter\checker\mathligtoks\mathlig@delim\fi
  \mathlig@next
 }%
 \mathlig@defcs{mathlig@toks@#1#2}{}%
 \mathlig@defcs{mathlig@chck@#1#2}##1##2\mathlig@delim{%
  \ifx\mathlig@next@tok##1%
   \mathlig@let@cs\mathlig@next@cmd{mathlig@look@#1#2##1}\let\mathlig@next\mathlig@gobble
  \fi
  \ifx\mathlig@delim##2\mathlig@delim\relax\else
   \csname mathlig@chck@#1#2\endcsname##2\mathlig@delim
  \fi
 }%
 \ifx\mathlig@delim#2\mathlig@delim\else
  \mathlig@defcs{mathlig@back@#1#2}{\csname mathlig@back@#1\endcsname #2}%
 \fi
}%

\catcode`@\mathlig@atcode

\mathchardef\ordinarycolon\mathcode`\:
\def\vcentcolon{\mathrel{\mathop\ordinarycolon}}
\mathlig{:=}{\vcentcolon=}
\mathlig{::=}{\vcentcolon\vcentcolon=}

\numberwithin{equation}{section}

\theoremstyle{plain}
\newtheorem{theo}{Theorem}[section]
\newtheorem{cor}[theo]{Corollary}
\newtheorem{lem}[theo]{Lemma}
\newtheorem{prop}[theo]{Proposition}

\newtheorem{ass}[theo]{Assumption}
\newtheorem{remme}[theo]{Remark}

\theoremstyle{definition}
\newtheorem{defn}[theo]{Definition}

\newtheorem*{theorem*}{Theorem}
\theoremstyle{remark}

\newtheorem*{ex*}{Example}
\theoremstyle{remark}
\newtheorem*{exs*}{Examples}
\theoremstyle{remark}
\newtheorem*{rems*}{Remarks}
\theoremstyle{remark}
\newtheorem*{rem*}{Remark}

\theoremstyle{definition}
\newtheorem*{idea*}{Idea}

\theoremstyle{remark}

\theoremstyle{remark}

\title[Spectrum of Exceptional Laguerre Operators]{The Spectrum of Self-Adjoint Extensions associated with Exceptional Laguerre Differential Expressions}

\author{Dale~Frymark}
\address{Department of Theoretical Physics, Nuclear Physics Institute, Czech Academy of Sciences, 25068 Řež, Czech Republic
}
\email{frymark@ujf.cas.cz}

\author{Jessica Stewart Kelly}
\address{Department of Mathematics, Christopher Newport University, 1 Avenue of the Arts, Newport News, VA 23606, USA.}
\email{jessica.kelly@cnu.edu}


\begin{document}

\begin{abstract}
Exceptional Laguerre-type differential expressions make up an infinite class of Schr\"odinger operators having rational potentials and one limit-circle endpoint.  In this manuscript, the spectrum of all self-adjoint extensions for a general exceptional Laguerre-type differential expression is given in terms of the Darboux transformations which relate the expression to the classical Laguerre differential expression. The spectrum is extracted from an explicit Weyl $m$-function, up to a sign.

The construction relies primarily on two tools: boundary triples, which parameterize the self-adjoint extensions and produce the Weyl $m$-functions, and manipulations of Maya diagrams and partitions, which classify the seed functions defining the relevant Darboux transforms. Several examples are presented. 
\end{abstract}

\maketitle

\setcounter{tocdepth}{1}
\tableofcontents

\section{Introduction}\label{s-intro}

First introduced in \cite{OS,OS2} in 2009, exceptional orthogonal polynomials (XOPs) describe a class of Sturm--Liouville polynomial families where some (exceptional) degrees of polynomials are missing.  These XOPs form eigenfunctions of rational Sturm--Liouville equations, but fall outside of the restrictions of the classical B\"ochner theorem \cite{B} that characterizes the classical orthogonal polynomial (COP) systems of Hermite, Laguerre and Jacobi.  The properties of XOP have been widely studied over the past 13 years \cite{GKM1,GKM2, GKM5, MQ} as this area has brought forth interesting connections between XOP families and their COP relatives as well as surprising observations such as the fact that although a finite number of degrees are omitted from the sequence of XOP solutions, each family is complete in the associated natural Hilbert space setting. 

The motivation of \cite{GKM1} to study these families was two fold---first, in response to a generalized B\"ochner problem posed in \cite{PT} and second, an interest in quantum mechanics. As such, there are numerous applications to mathematical physics. Perhaps most significantly, XOPs are, up to a gauge factor, eigenfunctions of exactly solvable potentials obtained by taking Darboux transforms of potentials, called rational extensions, from quantum mechanics \cite{G,MR,Q}. In particular, the harmonic oscillator, isotonic oscillator and trigonometric Darboux--P\"osch--Teller potential have rational extensions defined by Hermite, Laguerre and Jacobi polynomials, respectively. Many of these rational extensions are translationally shape invariant \cite{GGM2}, and their corresponding Darboux transforms are usually called supersymmetric quantum mechanical (SUSY QM) partnerships. 

Specific types of XOPs were initially introduced individually; well-known XOPs include Hermite XOPs \cite{GGM3}, Type I-III Laguerre XOPs \cite{LLMS} and Type I and II Jacobi XOPs \cite{LLK}. The spectral properties of the operators in each case have also been determined \cite{ALS, GGM3, LLK, LLMS, LLSW}. However, it was shown in \cite{GGM} that there are an infinite number of distinct XOP classes.  Although all XOP families are derived from a COP family, most of which do not fit into the finite pre-existing ``type'' classifications. In this manuscript, the focus will be on Laguerre-type XOP families and in particular, the associated spectral analysis.  

Generally speaking, every XOP family is related to a COP family by a sequence of Darboux transformations \cite{GKM3, GKM4, KLO, STZ}. When a suitable Darboux transformation is applied to a COP, the result is an XOP that is a eigenfunction for an exceptional eigenvalue problem.  Darboux transformations are often called ``state-deleting'' because when applied to an operator with a set of eigenfunctions (i.e.~an operator with COP solutions) they remove one or more of the eigenfunctions from this set. With regard to the spectrum, which consists only of eigenvalues, the term ``state-deleting'' is misleading as the set may remain the same or even include new points after a Darboux transform is performed.  It is also not immediately clear how the Darboux transformation impacts different self-adjoint extensions of the same symmetric differential expression; it is this affect of the Darboux transformation on the spectrum is what we aim to study. For instance, we ask does another extension compensate for these removed eigenfunctions in some way? Perturbation theory gives two immediate restrictions. First, because the spectrum of one self-adjoint extension is discrete, the spectrum of all self-adjoint extensions must be discrete, see e.g.~\cite{S2,S}. Second, the eigenvalues of extensions must change continuously with the parameterization of the self-adjoint extension.

To study the affect of the Darboux transformation on self-adjoint expressions, we exploit the intertwining nature of the COP and XOP expressions.  If the classical Laguerre expression is given in Schr\"odinger form, Darboux transforms can be classified via their changes to the potential thanks to the well-known Darboux--Crum formulas \cite{C}. These changes are given by a Wronskian whose entries are eigenfunctions of the COP expression. The process of determining the affect of the Darboux transformation on the spectrum requires manipulating general Wronskians into a standard form from which information can be extracted. Fortunately, this study may utilize shape invariance where many multi-step Darboux transformations can be shown to be equivalent up to a constant. The manuscripts \cite{BK} and \cite{GGM2} both give equivalence relations for Wronskians, and the methods used therein can be adapted for our purposes.

To exploit spectral information from the manipulated Wronskians, we use the framework of boundary triples.  This perspective is unique as compared to other spectral analyses of XOP families.  The theory of boundary triples provides a natural parameterization of all self-adjoint extensions in order to produce a Weyl $m$-function. The Weyl $m$-function associated with an operator provides all of the relevant spectral data. Here we are primarily focused on only the location of eigenvalues, but is is possible to extract the spectral measure and even eigenfunctions \cite{BFL}. The challenge of applying the theory in this context is that solutions to the general XOP must be normalized. This normalization simplifies to evaluating the transformed solution at $x=0$, which is achievable thanks to a standard form resulting from manipulation of the general Wronskians. Crucially, boundary triples for general XOP expressions are very easy to obtain, as they will carry the same structure as in the COP Laguerre expression with adjusted parameters.

\subsection*{Structure of the paper}\label{ss-structure}

The contents of this paper are as follows. Section \ref{s-prelims} reviews some essential facts about the formulation of XOP families via Darboux transforms as well as the basics of boundary triples and Weyl $m$--functions as they pertain to XOP expressions.  Maya diagrams are introduced as they are important to organizing information about the Wronskians associated with a Darboux transformation. Section \ref{s-type1} contains a relatively simple example, the Type I Laguerre XOP differential equation. Boundary triples are used to complete the spectral analysis of the Type I Laguerre XOP family of operators. This serves as a motivating example for the construction of the general methods in later sections. Note that some of the results in this example overlap with those of \cite{ALS,LLMS}. 

 A number of assumptions are made throughout the text as their need arises, but for the convenience of the reader we collect them here: Assumptions \ref{a-alpha}, \ref{a-partitions}, \ref{a-mayaorder}, \ref{a-signs} and \ref{a-type3}. Section \ref{s-mayaforxop} introduces the general notation and concepts necessary to discuss the seed functions and Wronskians that will define solutions of the XOP family. The Wronskians are defined by two Maya diagrams of seed functions: $M_1$ and $M_2$. This section also sets the parameters for the underlying Hilbert space of the XOP family in terms of these Maya diagrams.

Section \ref{s-manipulations} begins with the Frobenius analysis showing that the XOP expression has the limit-circle endpoint $x=0$ for similar parameter choices as the COP expression. Then we prove two central results: Theorems \ref{t-new4.2} and \ref{t-new4.22}. These theorems detail the shifting of the Maya diagrams from a general position into canonical and conjugate canonical positions, tracking constants and parameters along the way. The result of this shifting process is two solutions in standard form. Section \ref{s-normalizations} takes the Wronskians, which represent solutions of the XOP family, in these standard forms, and evaluates them at $x=0$. Effectively, dividing the solutions by these constants normalizes them with respect to the sesquilinear form. The resulting constants depend on many parameters, including $M_1$, $M_2$, $\al$ and the total number of seed functions in each position. Section \ref{s-mfunctions} combines the results of the previous two sections to properly augment the solutions of the XOP family in order to build a boundary triple. Construction of a general deficiency element then yields Weyl $m$-functions for each self-adjoint extension in the family, as shown in Corollary \ref{c-mfunctions}. A number of other observations and remarks are made in order to properly interpret the resulting expressions. Section \ref{s-example} illustrates the results of the manuscript by explicitly applying them to a previously unknown XOP family. The spectral properties of this family are surprising and, although any inverse spectral theory remains elusive, encourages future investigation. Lastly, a general summary of the work and closing remarks are made in Section \ref{s-conclusion}.

\section{Preliminaries}\label{s-prelims}

Generally, a XOP sequence $\{p_n\}_{\mathbb N_0\backslash A}$, where $p_n$ is a polynomial of degree $n$ and $A\subseteq \mathbb N_0$ is a non-empty finite set, satisfies the following properties:
\begin{enumerate}
    \item[(i)] each $y=p_n$ is a solution to the associated second-order eigenvalue problem;
    \item[(ii)] the associated eigenvalue problem admits no polynomial eigenfunctions having degree $n \in A$;
    \item[(iii)] the sequence $\{p_n\}_{\mathbb N_0\backslash A}$ is orthogonal on an open interval with respect to a positive measure; and
    \item [(iv)] all moments $\{\mu_n\}_{n=0}^\infty$ exist and are finite.  
\end{enumerate} Each XOP expression may be written as the composition of first-order operators or Darboux transformations; these first-order operators are derived from COP expressions.  We outline the process below.

Suppose $\ell$ is a second-order expression of the form \begin{equation}\label{e-express}
    \ell[y](x)=p(x)y''(x)+q(x)y'(x)+r(x)
\end{equation} whose associated eigenvalue equation is \begin{equation}\label{e-eigenvalue}
    \ell[y]=\lambda y,
\end{equation} where $\lambda\in \mathbb C$. 
Define rational functions
\begin{align*}
    P(x)&=\exp\left(\int \frac{q(z)}{p(z)}dz\right),\\
    W(x)&=\frac{P(x)}{p(x)}, \quad \mbox{ and }\\
    R(x)&=r(x)W(x).
\end{align*} Multiplying Eq.~\eqref{e-eigenvalue} by $W(x)$ yields the associated symmetric Sturm-Liouville form of Eq.~\eqref{e-express}
\begin{equation}
    P(y')'+Ry=\lambda Wy.
\end{equation} It follows that $W(x)$ is the weight function associated with the expression $\ell$.

Two second-order expressions $\ell$ and $\cL$ are \textit{gauge-equivalent} if there exists a rational function $\sigma$ such that 
\[\cL=\sigma \ell \sigma^{-1}.\]


In order to derive the XOP operator from its COP counterpart, we first need to rewrite $\ell$ as a composition of two first-order operators.  A quasi-rational function $\phi$ is a \textit{quasi-rational eigenfuction} for the expression $\ell$ if $\ell[\phi]=\lambda \phi$ for $\lambda\in \mathbb C$. For example, in the case of the Laguerre COP, the quasi-rational eigenfunctions that produce XOP systems are:
\begin{align}\label{e-phi1}
    \phi_1(x)&=L_m^\alpha(x)\\\label{e-phi2}
    \phi_2(x)&=e^xL_m^\alpha(-x)\\\label{e-phi3}
    \phi_3(x)&=x^{-\alpha}L_m^{-\alpha}(x)\\\label{e-phi4}
    \phi_4(x)&=e^xx^{-\alpha}L_m^\alpha(-x).
\end{align} The quasi-rational eigenfunction Eq.~\eqref{e-phi3} produces the Type I Laguerre XOP family. The reader is directed to Section \ref{s-example} for more details into the decomposition of the COP Laguerre expression and formulation of the Type I expression.  

\begin{prop}\cite[Proposition 3.5]{GGM} For a second-order differential operator $\ell[y]$ having rational coefficients, let $\phi$ be a quasi-rational eigenfunction of $\ell$ with eigenvalue $\lambda$, and let $b(x)$ be an arbitrary, non-zero rational function. Define rational functions
\begin{align*}
    w&=\frac{\phi'}{\phi},\\
    \widehat{b}&=\frac{p}{b},\quad \mbox{ and}\\
    \widehat{w}&=-w-\frac{q}{p}+\frac{\widehat{b}}{b}
\end{align*}
and first-order operators $A$ and $B$ by
\begin{equation}\label{e-operators}
    A[y] = b(y'-wy) \quad \mbox{ and }\quad B[y] = \widehat{b}(y'- \widehat{w}y).
\end{equation}
\end{prop}

The operators $A$ and $B$ of Eq.~\eqref{e-operators} are used to form  a \textit{rational factorization} 
\begin{align*}
    \ell=BA+\la_0,
\end{align*} of $\ell$.  Given a rational factorization, the \textit{partner operator} $\cL$ is defined via
\begin{align*}
    \cL=AB+\la_0.
\end{align*} The operator $\cL$ will be a second--order operator of the form
\begin{equation}\label{e-XOPop}
    \cL[y](x)=p(x)y''(x)+\cQ(x)y'+\cR(x) y(x)
\end{equation} where $\cQ$ and $\cR$ are rational functions defined by
\begin{align*}
    \cQ&=q+p'-\frac{2pb'}{b},\\
    \cR&=-p(\widehat{w}'+\widehat{w}^2)-\cQ\widehat{w}\\
    &=r+q'+wp'+\frac{b'}{b}(q+p')+\left(2\left(\frac{b'}{b}\right)^2-\frac{b''}{b}+2w'\right)p,\quad \mbox{ and}\\
    \cW&=\frac{pW}{b^2}=\frac{P}{b^2}.
\end{align*}

When $\sigma$ and $\phi$ are suitably chosen and $\ell$ is a COP expression, the partner operator $\cL$ produces an XOP operator having weight function $\cW$.  The reader is directed to \cite[Proposition 2.5, Theorem 5.4]{GGM} for the suitability conditions on $\sigma$.  The polynomial part of $\phi$ is important for future calculations and is a generalized Laguerre polynomial in the context of Laguerre XOP families, see Eq.~\eqref{e-particularlaguerre}.

The transformation $\cA:=\ell\to \cL$ is a (rational) \textit{Darboux transformation}.  In fact, with appropriate assumptions, this process may be iterated $n \in \mathbb N$ times to produce a factorization chain or a Darboux transformation of size $n$.  These complex Darboux transformations are used to generate XOP systems.  We emphasize that factorizations of a COP are not unique; and herein provides the opportunity for many XOP families.

One crucial property that arises from this construction is the following:
\begin{lem}[Intertwining Property]{\cite[Proposition 3.2]{GGM}}\label{l-intertwine} 
Let $\ell$ be a COP expression such that $\ell[y]=\la y$ for some $\la\in\CC$ and $\cL$ be an XOP expression. If $\ell$ and $\cL$ are related via a rational Darboux transformation then the following intertwining relations hold:
\begin{align*}
A\ell=\cL A, \hspace{2em} \ell B=B\cL.
\end{align*}
\end{lem}
In particular, if $y_\la$ is an eigenfunction of $\ell[y]=\lambda y$, then $\wt{y}_\la:=A[y_\la]$ satisfies $\cL[\wt{y}_\la]=\la\wt{y}_\la$.

Some final observations in regards to notation---to simplify our discussion, we will refer to all self-adjoint extensions of a symmetric operator as a ``family.'' XOP and COP families thus refer to families that contain one extension which possesses a set of XOPs or COPs. The notation $\ell$ and $\cL$ will often be used to denote the classical and exceptional symmetric differential expressions, respectively, that generate families.  Also, the general form of COP expressions is dependent upon a parameter.  For the Laguerre expression, this parameter is denoted by $\alpha$.  It follows that both $\ell$ and $\cL$ also have dependency upon this parameter.  When emphasis on the parameter is needed, we will use $\ell^\al$ and $\cL^\al$ for $\ell$ and $\cL$, but for ease of notation, the parameter will often be suppressed. Additionally, when relevant, a superscript of I, II or III may be included to indicate the XOP Laugerre type.

\subsection{Boundary Triples}\label{ss-bts}

The principal tools for extracting spectral information from self-adjoint extensions of symmetric differential operators in this manuscript are the Weyl $m$-functions that are generated by boundary triples and the parameterization of all self-adjoint extensions. For completeness, some basics from the theory of boundary triples are included here. The content of this subsection is also found in \cite{BdS} and \cite{BFL}, which the interested reader may consult for more details. For the purposes of this paper, we restrict ourselves to the simplified case of Sturm--Liouville differential operators rather than more general linear relations.  Additionally, we further narrow the scope of discussion as the XOP Laguerre differential expressions will have deficiency indices $(1,1)$, see Section \ref{s-mayaforxop}. 

Let $\cL\ti{min}$ denote an exceptional Laguerre operator on $\cH=L^2[(a,b),\;\cW(x)dx]$, $(a,b)\subset\RR\cup\{\pm\infty\}$, $\cW(x)>0$ a.e.~in $(a,b)$ and $\cW(x)\in L^1\ti{loc}[(a,b),dx]$. Denote the associated maximal operator $\cL\ti{max}$, domain $\cD\ti{max}$ and sesquilinear form $[\cdot,\cdot]\ci\cL$. 

The sesquilinear form associated to the expression $\cL$ can be explicitly calculated for $f,g\in\cD\ti{max}$ as
\begin{align}\label{e-xopsesqui}
[f,g]\ci\cL(x)=p(x)\left[f(x)g'(x)-f'(x)g(x)\right].
\end{align}

\begin{theo}[{\cite[Section 17.2]{N}}]\label{t-limits}
The limits $[f,g](b):=\lim_{x\to b^-}[f,g](x)$ and $[f,g](a):=\lim_{x\to a^+}[f,g](x)$ exist and are finite for $f,g\in\cD\ti{max}$.
\end{theo}

\begin{defn}{\cite[Definition 2.1.1]{BdS}}\label{d-bt}
A \textit{boundary triple} for $\cL\ti{max}$, denoted $\{\RR,\Gamma_0,\Gamma_1\}$, is composed of the boundary space $\RR$ and two surjective linear maps $\Gamma_0,\Gamma_1:\cD\ti{max}\to\RR$ for which the Green's identity 
\begin{align}\label{e-btgreens}
    \langle \cL\ti{max}f,g\rangle_{\cH}-\langle f,\cL\ti{max}g\rangle_{\cH}=\langle\Gamma_1f,\Gamma_0g\rangle_{\RR}-\langle\Gamma_0f,\Gamma_1g\rangle_{\RR},
\end{align}
holds for all $f,g\in \cD\ti{max}$.
\end{defn}
Notice that the left-hand side of Eq.~\eqref{e-btgreens} is simply $[\cdot,\cdot]\ci\cL$.

\begin{defn}{\cite[Definition 1.4.9]{BdS}}
Let $\zeta\in\CC$. The space
\begin{align*}
    \mathfrak{N}_{\zeta}(\cL\ti{max}):=\ker(\cL\ti{max}-\zeta),
\end{align*}
is called the {\it defect subspace} of $\cL\ti{min}$ at the point $\zeta\in\CC$.
\end{defn}
The usual positive and negative defect spaces are then simply $\mathfrak{N}_i(\cL\ti{max})$ and $\mathfrak{N}_{-i}(\cL\ti{max})$, respectively.  It should be clear that boundary triples are not typically unique. Indeed, given a self-adjoint extension $H$ of $\cL\ti{min}$, $\cL\ti{max}$ can be decomposed into a direct sum of $\dom(H)$ and a defect space. This allows maps $\Gamma_0$ and $\Gamma_1$ to be defined so that $\dom(H)=\ker(\Gamma_0)$ and $\{\CC,\Gamma_0,\Gamma_1\}$ form a boundary triple, see \cite[Theorem 2.4.1]{BdS}. 

In our context, the boundary triples for $\cL\ti{max}$ will be formed with quasi-derivatives.

\begin{defn}\label{d-quasideriv}
Let $u$ and $v$ be linearly independent real solutions of the equation $(\cL\ti{max}-\zeta_0)y=0$ for some $\zeta_0\in\RR$ and assume that the solutions are normalized by $[u,v]\ci\cL=1$. For $f\in\cD\ti{max}$, the quasi-derivatives of $f$ are induced by the normalized solutions $u$, $v$ and defined as complex functions on $(a,b)$ given by
\begin{align*}
f^{[0]}:=[f,v]_L \text{ and } f^{[1]}:=-[f,u]_L.
\end{align*}
\end{defn}

Note that in practice, the minus sign outside the sesquilinear form in the first quasi-derivative is absorbed into the definition of $u$. These quasi-derivatives naturally define two maps and self-adjoint extensions for exceptional Laguerre operators:
\begin{equation}\label{e-generalsetup}
\begin{aligned}
    \Gamma_0f&:=f^{[0]}(0), \hspace{3em} \cL_\infty:=\{f\in \cL\ti{max}:~f\in\Ker(\Gamma_0)\}, \\
    \Gamma_1f&:=f^{[1]}(0), \hspace{3em} \cL_0:=\{f\in \cL\ti{max}:~f\in\Ker(\Gamma_1)\}.
\end{aligned}
\end{equation}
Recall that for all $f\in\cD\ti{max}$ the quasi-derivatives $f^{[0]}(0)$ and $f^{[1]}(0)$ are well-defined due to Theorem \ref{t-limits}. All self-adjoint extensions of $\cL\ti{min}$ are in one-to-one correspondence with 
\begin{align}\label{e-ltau}
    \cL_\tau:=\{f\in \cL\ti{max}:~\tau\Gamma_0f=\Gamma_1f\},
\end{align}
where $\tau\in\RR\cup\{\infty\}$. The case $\tau=\infty$ is interpreted as representing the operator $\cL_\infty$; this case technically represents the linear relation $\{0,\RR\}$. The only other one-dimensional linear relation that isn't an operator, $\{\RR,0\}$, clearly corresponds to $L_0$. The definitions can be made rigorous by appealing to the corresponding semi-bounded forms of the operators; interested readers should consult \cite[Remark 3.7]{BFL}. The above definitions are adequate for our purposes. 

Also note that the extension $\cL_0$ is identified as the Friedrichs extension when $u$ is the principal solution by \cite{MZ}: the extension whose semi-bounded form has the greatest lower bound. On the other hand, if $\cL\ti{max}$ is positive and $v$ is a non-principal solution, $\cL_\infty$ is the Krein--von Neumann extension, see \cite[Definition 5.4.2]{BdS}. This structure immediately allows for the definition of a Weyl $m$-function.

\begin{defn}{\cite[Definition 2.3.1, 2.3.4]{BdS}}\label{d-mfunction}
Let $\{\RR,\Gamma_0,\Gamma_1\}$ be a boundary triple for $\cL\ti{max}$ and $\la\in\CC$. Then
\begin{equation*}
    \rho(\cL_\infty)\ni\la\mapsto M_\infty(\la)=\Gamma_1\left(\Gamma_0\upharpoonright\mathfrak{N}_{\la}(\cL\ti{max})\right)^{-1},
\end{equation*}
where $\upharpoonright$ denotes the restriction, is called the \textit{Weyl $m$-function} associated with the boundary triple $\{\RR,\Gamma_0,\Gamma_1\}$.
\end{defn}

In this context, the spectrum of $\cL_{\infty}$ is discrete and the difference of the resolvents of $\cL_{\infty}$ and $\cL_\tau$ is an operator of rank one. Thus the spectrum of the self-adjoint operator $\cL_\tau$ is also discrete. Indeed, $\la\in\rho(\cL_\infty)$ is an eigenvalue of $\cL_\tau$ if and only if $\ker (\tau-M_\infty(\la))$ is nontrivial. For $\la\in\rho(\cL_0)\cap\rho(\cL_\infty)$, the spectral properties of $\cL_\tau$ can thus be described with the help of the function 
\begin{align}\label{e-thetam}
    M_\tau(\la)=(\tau-M_\infty(\la))^{-1},
\end{align}
 see \cite[Equation (3.8.7)]{BdS}. The poles of the function \eqref{e-thetam} coincide with the discrete spectrum of $L_\tau$ and the dimension of the eigenspace $\ker(L_\tau-\la)$ coincides with the dimension of the range of the residue of the function \eqref{e-thetam} at $\la$. The relationship $M_0=-M_\infty^{-1}$ is immediately apparent. Note that $M_0(\la)$ can be obtained without exploiting Eq.~\eqref{e-thetam} by simply switching the definitions of the maps $\Gamma_0$ and $\Gamma_1$. A similar construction that the reader may find useful for more general Sturm--Liouville operators is available in \cite{BFL, F}.

Now, fix a fundamental system $(u_1(\cdot,\la);~u_2(\cdot,\la))$ for the equation $(\cL\ti{max}-\la)f=0$ by the initial conditions
\begin{align}\label{e-initialconditions}
    \left( \begin{array}{cc}
u_1^{[0]}(0,\la) &  u_2^{[0]}(0,\la) \\
u_1^{[1]}(0,\la) & u_2^{[1]}(0,\la)
\end{array} \right)
=
    \left( \begin{array}{cc}
1 & 0 \\
0 & 1
\end{array} \right).
\end{align}

\begin{prop}{\cite[Proposition 6.4.9]{BdS}}\label{p-btsetup}
Let $\Gamma_0$ and $\Gamma_1$ be defined as in Eq.~\eqref{e-generalsetup}. Then $\{\RR,\Gamma_0,\Gamma_1\}$ is a boundary triple for $\cD\ti{max}$. Moreover, if $\la\in\CC\backslash\RR$ and $\chi(x,\la)$ is a nontrivial element in $\mathfrak{N}_\la(\cL\ti{max})$, then $\chi^{[0]}(0,\la)\neq0$ and the Weyl $m$-function is given by
\begin{align*}
    M_\infty(\la)=\frac{\chi^{[1]}(0,\la)}{\chi^{[0]}(0,\la)}.
\end{align*}
\end{prop}

It is then possible to transform the boundary triple so that $M_\tau(\la)$ naturally emerges from Proposition \ref{p-btsetup} by using \cite[Equation 6.4.8]{BdS}:
\begin{align}\label{e-generalbt}
M_\tau(\la)=\frac{1+\tau M_\infty(\la)}{\tau-M_\infty(\la)}, \hspace{2em}\la\in\CC\backslash\RR.
\end{align}

This means that a corresponding Weyl $m$-function can be accessed for each self-adjoint extension in the XOP family, and a spectral analysis can be completed by finding the poles of these $m$-functions. 

\subsection{Maya Diagrams}\label{ss-diagrams}

Maya diagrams will be used to indicate the indices of seed functions within the Wronskians that define Darboux transformations. These Maya diagrams are vitally important to the construction of XOPs and will be manipulated throughout the manuscript.  We include a brief overview here that mostly follows \cite{BK}. More information on Maya diagrams, Young diagrams and partitions relating to XOP families may be found in \cite{BK, GGM2}.

A \textit{Maya diagram} $M$ is a infinite subset of integers that contains finitely many non-negative integers and excludes finitely many negative integers. Visually, imagine an infinite row of boxes with a fixed origin. These boxes are either filled (the integer belongs to $M$) or empty (the integer does not belong to $M$).  Thus, $M$ is characterized by two finite sequences of integers: the positive integers included in $M$ and the negative integers excluded by $M$.

Each of the non-negative integers $n_i$ in $M$ are labeled to make a decreasing sequence by 
\begin{align*}
    n_1>n_2>\dots>n_{r_1}\geq 0,
\end{align*}
where $r_1$ is the number of filled boxes to the right of the origin.  If no boxes are filled to the right of the origin, the sequence is empty and $r_1=0$. Likewise, the number of empty boxes to the left of the origin can be denoted as $r_4$, the negative integers $k$ corresponding to these boxes assigned values via $n'=-k-1$ and
\begin{align*}
    n_1'>n_2'>\dots>n_{r_4}'\geq 0.
\end{align*}
The Maya diagram can then be succinctly described by these two sequences as
\begin{align*}
    M=(n_1',n_2',\dots,n_{r_4}'~|~n_1,n_2,\dots,n_{r_1}).
\end{align*}

As an example, let $M$ be the  Maya diagram $M=(5,2,1~|~4,3,1)$, whose graphical representation may be seen in Figure \ref{f-exMaya}.
	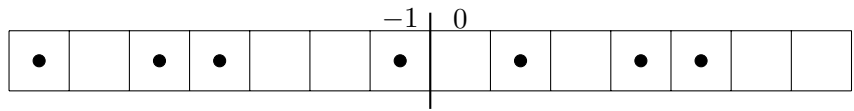
\begin{figure}[H]\label{f-exMaya}\caption{Maya Diagram }
		\begin{tikzpicture}[scale=.8]
			\draw (0,0) grid(14,1);
			\draw [fill] (.5,0.5) circle [radius=0.1];
			\draw [fill] (2.5,0.5) circle [radius=0.1];
			\draw [fill] (3.5,0.5) circle [radius=0.1];
			\draw [fill] (6.5,0.5) circle [radius=0.1];
			\draw [fill] (8.5,0.5) circle [radius=0.1];
			\draw [fill] (10.5,0.5) circle [radius=0.1];
			\draw [fill] (11.5,0.5) circle [radius=0.1];
			\draw [thick] (7,-.3)--(7,1.3);
			\node at (6.5,1.2) {$-1$};
			\node at (7.5,1.2) {$0$};
			\node at (6.5, -1){$M=(5,2,1~|~4,3,1)$};
		\end{tikzpicture}\hfill
	\end{figure}
All boxes to the left of those shown in Figure \ref{f-exMaya} are filled, while those to the right are empty. The origin of a Maya diagram may be shifted to the left or right by adding or subtracting, respectively. The Maya diagram $\wt{M}=M+3$ is represented by Figure \ref{f-exShiftMaya}.
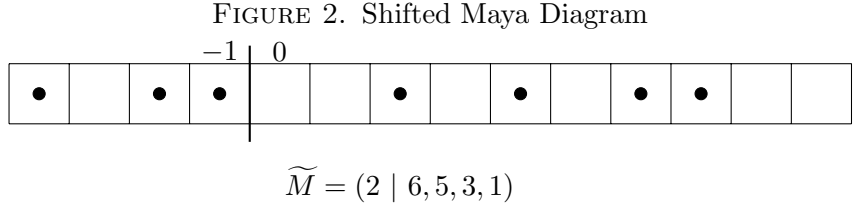
\begin{figure}[H]\label{f-exShiftMaya}\caption{Shifted Maya Diagram}
		\begin{tikzpicture}[scale=.8]
			\draw (0,0) grid(14,1);
			\draw [fill] (.5,0.5) circle [radius=0.1];
			\draw [fill] (2.5,0.5) circle [radius=0.1];
			\draw [fill] (3.5,0.5) circle [radius=0.1];
			\draw [fill] (6.5,0.5) circle [radius=0.1];
			\draw [fill] (8.5,0.5) circle [radius=0.1];
			\draw [fill] (10.5,0.5) circle [radius=0.1];
			\draw [fill] (11.5,0.5) circle [radius=0.1];
			\draw [thick] (4,-.3)--(4,1.3);
			\node at (3.5,1.2) {$-1$};
			\node at (4.5,1.2) {$0$};
			\node at (6.5, -1){$\wt{M}=(2~|~6,5,3,1)$};
		\end{tikzpicture}\hfill
	\end{figure}
In particular, the length of the two subsequences, $r_1+r_4$, is not necessarily stable under such translations: $M$ has $6$ indices and $\wt{M}$ has $5$. 

A Maya diagram is in \textit{canonical form} if there are no empty boxes to the left of the origin, but the first box to the right is empty. This form is useful for comparing diagrams and can always be obtained in the current context by applying a shift. If $M+t$ puts a Maya diagram $M$ into canonical form for some $t\in \mathbb Z$, we denote this shift as $t=t(M)$.  In the context of the above example, $t(M)=6$, or $n_1'+1$.

Alternatively, the \textit{conjugate canonical form} of a Maya diagram is the position when there are no filled boxes to the right of the origin, but the first box to the left is filled. The required shift to put $M$ into conjugate canonical form is denoted by $t'=t'(M)$. In the above example, $t'(M)=-(n_1+1)=-5$ and we have the associated representation in Figure \ref{f-exCCongMaya}.
\begin{figure}[H]\label{f-exCCongMaya} \caption{Conjugate Canonical Form Maya Diagram }
		\begin{tikzpicture}[scale=.8]
			\draw (0,0) grid(14,1);
			\draw [fill] (.5,0.5) circle [radius=0.1];
			\draw [fill] (2.5,0.5) circle [radius=0.1];
			\draw [fill] (3.5,0.5) circle [radius=0.1];
			\draw [fill] (6.5,0.5) circle [radius=0.1];
			\draw [fill] (8.5,0.5) circle [radius=0.1];
			\draw [fill] (10.5,0.5) circle [radius=0.1];
			\draw [fill] (11.5,0.5) circle [radius=0.1];
			\draw [thick] (12,-.3)--(12,1.3);
			\node at (11.5,1.2) {$-1$};
			\node at (12.5,1.2) {$0$};
			\node at (6.5, -1){$M-5=(10,7,6,4,2~|~\emptyset)$};
		\end{tikzpicture}\hfill
	\end{figure}
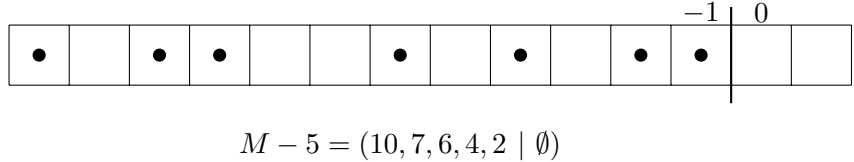

One of the goals of Section \ref{s-manipulations} will be to manipulate the Maya diagrams $M_1$ and $M_2$ of Eqs.~\eqref{e-generalizedmaya1} and \eqref{e-generalizedmaya2} so that they are either in canonical or conjugate canonical position and can therefore each be described by a single finite decreasing sequence of non-negative integers. Naturally, such sequences describe a partition, see Eqs.~\eqref{e-canonicalpartition1} and \eqref{e-canonicalpartition2}. Hence, any given Maya diagram gives rise to a partition, when shifted into canonical form, and a conjugate partition, when shifted into conjugate canonical form.

\section{Motivating Example}\label{s-type1} 
We begin with an example: the Type I Laguerre exceptional family. The spectral properties of the extension $\cL_0$ are already known from \cite{LLMS}, but the authors find it instructive to present the construction of the boundary triple and $m$-functions in this specific context before moving on to the general setting. 

The classical Laguerre differential expression $\ell^{\al}$ acts via
\begin{align}\label{e-laguerre}
    \ell^{\al}[f](x):=-xf''(x)+(x-\al-1)f'(x)\,.  
\end{align} For fixed $m\in \NN$, $\ell^{\al}[f](x)$ has the rational factorization
\begin{align}\label{e-laguerrefac}
-\ell^{\al}&=B^{I,\al}_m\circ A^{I,\al}_m+\al+m+1,
\end{align}
where
\begin{align}\label{e-Adarboux}
    A^{I,\al}_m[f]&:=L_m^{\al}(-x)f'(x)-L_m^{\al+1}(-x)f(x),\\ B^{I,\al}_m[f]&:=\frac{xf'(x)+(1+\al)f(x)}{L^{\al}_m(-x)},
\end{align} and $L_m^\al$ is the classical Laguerre polynomial of degree $m$.  The rational factorization in Eq.~\eqref{e-laguerrefac} is used to defined the differential expression $\cL^{I,\al}_m$ for the Type I Laguerre XOP operator
\begin{equation}\label{e-ratfact1}
    \mathcal L^{I,\al}_m=-(A^{I,\al-1}_m\circ B^{I,\al-1}_m+\al+m).
\end{equation} In this case, $\cL^{I,\al}_m[y]=\lambda y$ contains no polynomial solutions of degrees $0, 1,\ldots, m-1$.  The weight function associated with the XOP family is 
\begin{equation}
     \cW_m^{I,\al}:=\frac{x^{\al}e^{-x}}{\left[L_m^{\al-1}(-x)\right]^2}, 
\end{equation} and the associated maximal domain in the Hilbert space $L^2\left[(0,\infty),\cW^{I,\al}_m\right]$ is defined as
\begin{align*}
    \cD_{\ti{max}}:=\big\{f:(0,\infty)\to\CC\big|f,f'&\in AC\ti{loc}(0,\infty); f,\cL^{I,\al}_m[f]\in L^2\left[(0,\infty),W^{I,\al}_m\right]\big\}.
\end{align*}
For $f,g\in\cD_{\ti{max}}$ and $x\in(0,\infty)$, the sesquilinear form is given by 
\begin{align*}
    [f,g]\ci\cL(x):=\frac{x^{\al+1}e^{-x}}{[L_m^{\al-1}(-x)]^2}\left[f(x)g'(x)-f'(x)g(x)\right].
\end{align*}
Frobenius analysis shows the endpoint $x=0$ is limit-circle for $\cL^{I,\al}_m$ when $\al\in(0,1)$; and the endpoint $x=\infty$ is limit-point for $\al>0$. Observe that the functions 
\begin{align*}
\wt{y}_1:=-\frac{[L_m^{\al-1}(0)]^2}{\al} \hspace{.5em}\text{ and }\hspace{.5em} \wt{y}_2:=x^{-\al},
\end{align*}
are both in $L^2\left[(0,\infty),\cW^{I,\al}_m\right]$ and are particular solutions to $\cL^{I,\al}_m f=0$. Together they define maps using the quasi-derivatives:
\begin{align*}
    \Gamma_0f&:=f^{[0]}(0)=[f,\wt{y}_2]\ci\cL(0)=\lim_{x\to0^+}-\frac{\al f(x)+xf'(x)}{[L_m^{\al-1}(0)]^2}
\end{align*}
and \begin{align*}
    \Gamma_1f&:=f^{[1]}(0)=[f,\wt{y}_1]\ci\cL(0)=\lim_{x\to0^+}\frac{x^{\al+1}}{\al}f'(x).
\end{align*}
Using results of Subsection \ref{ss-bts}, it is easy to show that $\{\Gamma_0,\Gamma_1,\CC\}$ is a boundary triple for $\cD_{\ti{max}}$. Note that the choice of $\wt{y}_1$ immediately yields $[\f,x^{-\al}]\ci\cL(0)=1$. The boundary triple naturally defines two self-adjoint extensions $\cL_0$ and $\cL_\infty$ (suppressing the dependence on $m$ and $\al$) that act via $\cL_m^{I,\al}$ on the domains
\begin{equation}\label{e-specialexts}
\begin{aligned}
\dom(\cL_0)&=\left\{f\in\Delta_m^{I,\al}~:~\Gamma_1(f)=0\right\}, \text{ and} \\
\dom(\cL_\infty)&=\left\{f\in\Delta_m^{I,\al}~:~\Gamma_0(f)=0\right\},
\end{aligned}
\end{equation}
respectively. 

The two linearly independent solutions of the Laguerre differential equation $\ell^{\al-1}$ are given by confluent hypergeometric functions:
\begin{align}\label{e-FS1}
    M(-\la,\al,x)= \,_1F_1(-\la,\al,x)\hspace{2em}(\text{for }\al+1\notin-\NN_0), \\\label{e-FS2}
    x^{-\al+1}M(1-\la-\al,2-\al,x)\hspace{2em}(\text{for }\al+1\notin\NN>1).
\end{align}
As a consequence of Lemma \ref{l-intertwine}, the Type I Laguerre XOP expression is therefore intertwined with the classical Laguerre expression $\ell^{\al-1}$. Therefore a fundamental system for the equation $(\cL^{I,\al}_m-\la)f=0$ can be found by plugging the solutions of Eqs.~\eqref{e-FS1} and \eqref{e-FS2} into the Darboux transformation $-A^{I,\al-1}$ of Eq.~\eqref{e-Adarboux}.  Thus
\begin{align*}
    \Omega_{1,m}^{\al}(x,\la)&:=-A^{I,\al-1}[M(-\la,\al,x)] \\   &=\frac{\Gamma(m+\al)}{\Gamma(m+1)\Gamma(\al)}\Big[\frac{\al+m}{\al}M(-m,\al+1,-x)M(-\la,\al,x) \\
    & \hspace{5em} +\frac{\la}{\al}M(-m,\al,-x)M(-\la+1,\al+1,x)\Big],
\end{align*}
and
\begin{align*}
    \Omega_{2,m}^{\al}(x,\la)&:=-A^{I,\al-1}[x^{-\al+1}M(1-\la-\al,2-\al,x)]\\
    &=\frac{\Gamma(m+\al+1)}{\Gamma(m+1)\Gamma(\al+1)}x^{-\al} \\
    &\hspace{3em}\times\Bigg[\frac{\al(1-\al)}{m+\al}M(-m,\al,-x)M(1-\la-\al,1-\al,x) \\
    & \hspace{6em}+xM(-m,\al+1,-x)M(1-\la-\al,2-\al,x)\Bigg]
\end{align*}
are two linearly independent solutions of the XOP expression $(\cL^{I,\al}_m-\la)f=0$. Furthermore, a straightforward calculation shows
\begin{align}\label{e-iclaguerre}
    \left( \begin{array}{cc}
 \left(\Omega_{1,m}^{\al}\right)^{[0]}(0,\la) &  \left(\Omega_{2,m}^{\al}\right)^{[0]}(0,\la) \\
 \left(\Omega_{1,m}^{\al}\right)^{[1]}(0,\la) & \left(\Omega_{2,m}^{\al}\right)^{[1]}(0,\la) 
\end{array} \right)
=
    \left( \begin{array}{cc}
1/C^{\al}_m & 0 \\
0 & 1/D^{\al}_m
\end{array} \right), 
\end{align}
where 
\begin{align*}
    C^{\al}_m(\la)&:=-\frac{\Gamma(m+\al)}{(\la+m+\al)\Gamma(\al)\Gamma(m+1)} \hspace{.5em}\text{ and }\hspace{.5em} D^{\al}_m(\la):=-\frac{\Gamma(m+1)\Gamma(\al)}{(1-\al)\Gamma(m+\al)}.
\end{align*}

The fundamental system composed of $u_1(x,\la)$ and $u_2(x,\la)$ must satisfy both $(\cL^{I,\al}_m-\la)f=0$ and the initial conditions
\begin{align*}
    \left( \begin{array}{cc}
 u_1^{[0]}(0,\la) &  u_2^{[0]}(0,\la) \\
 u_1^{[1]}(0,\la) & u_2^{[1]}(0,\la)
\end{array} \right)
=
    \left( \begin{array}{cc}
1 & 0 \\
0 & 1
\end{array} \right).
\end{align*}
If we suppress the dependence on $\al$ and $m$, this can be accomplished by setting $u_1(x,\la):=C^{\al}_m(\la)\Omega_{1,m}^{\al}(x,\la)$ and $u_2(x,\la)=D^{\al}_m(\la)\Omega_{2,m}^{\al}(x,\la)$.

It remains only to compute the explicit Weyl $m$-function and extract spectral information for each self-adjoint extension. A general deficiency element satisfying $(\ell^{\al-1}-\la)f=0$ and denoted by $\chi(x,\la)$ may be written via the Tricomi confluent hypergeometric function \cite[Eq.~(13.2.42)]{DLMF} as
\begin{align*}
U(-\la,\al,x)=\frac{\Gamma(1-\al)}{\Gamma(1-\al-\la)}M(-\la,\al,x)+\frac{\Gamma(\al-1)}{\Gamma(-\la)}x^{-\al+1}M(1-\la-\al,2-\al,x).
\end{align*}
The intertwining property in Lemma \ref{l-intertwine} implies that a general deficiency element $\wt{\chi}$ of $(\cL^{I,\al}_m-\la)f=0$ can be written as
\begin{align*}
\wt{\chi}(z,\la)&=-A^{I,\al-1}[U(-\la,\al,x)]\\
& =\frac{\Gamma(1-\al)}{\Gamma(1-\al-\la)}\left(-A^{I,\al-1}[M(-\la,\al,x)]\right) \\
& \hspace{8em}+\frac{\Gamma(\al-1)}{\Gamma(-\la)}\left(-A^{I,\al-1}[x^{-\al+1}M(1-\la-\al,2-\al,x)]\right)\\
&=\frac{\Gamma(1-\al)}{\Gamma(1-\al-\la)}\Omega_{1,m}^{\al}(x,\la)+\frac{\Gamma(\al-1)}{\Gamma(-\la)}\Omega_{2,m}^{\al}(x,\la)\\
&=\frac{\Gamma(1-\al)}{C^{\al}_m(\la)\Gamma(1-\al-\la)}u_1(x,\la)+\frac{\Gamma(\al-1)}{D^{\al}_m\Gamma(-\la)}u_2(x,\la).
\end{align*}
Hence, the initial conditions mean
\begin{align*}
    \chi\ci\cL^{[0]}(0,\la)=\frac{\Gamma(1-\al)}{C^{\al}_m(\la)\Gamma(1-\al-\la)} \hspace{.5 in}\text{ and }\hspace{.5 in}
    \chi\ci\cL^{[1]}(0,\la)=\frac{\Gamma(\al-1)}{D^{\al}_m\Gamma(-\la)}.
\end{align*}
Proposition \ref{p-btsetup} implies that if $\la\in\rho(\cL_\infty)$ the Weyl $m$-function for the extension $L_\infty$ is given by 
\begin{align*}
M_\infty(\la)&=\frac{\chi\ci\cL^{[1]}(0,\la)}{\chi\ci\cL^{[0]}(0,\la)}=\left(\frac{\Gamma(m+\al)}{\Gamma(m+1)\Gamma(\al)}\right)^2\frac{(1-\al)\Gamma(\al-1)\Gamma(1-\al-\la)}{(\la+m+\al)\Gamma(-\la)\Gamma(1-\al)}\\
&=-\left[L_m^{\al-1}(0)\right]^2\frac{\Gamma(\al)\Gamma(1-\al-\la)}{(\la+m+\al)\Gamma(-\la)\Gamma(1-\al)}.
\end{align*}
The spectrum of the self-adjoint operator $\cL_\infty$ are those points which are poles of $M_\infty(\la)$. Recall that the Gamma function has no zeros, but does have simple poles at zero and the negative integers.  Therefore, $M_\infty$ has poles at $\sigma(\cL_\infty)=(-m-\al)\cup \left\{n+1-\al\right\}_{n\in\NN_0}$ and  $\sigma(\cL_\infty)$ is the spectrum of the self-adjoint operator $\cL_\infty$.  The corresponding self-adjoint extension of the COP expression has the same spectrum with the point $-m-\al$ replaced by $-\al$. 

Likewise, the $m$-function of $\cL_0$, denoted $M_0$, can be found by computing $-1/M_\infty(\la)$ for $\la\in\rho(\cL_\infty)\cup\rho(\cL_0)$. This extension contains the Type I Laguerre XOP polynomials and is easily seen to have eigenvalues $\la=\left\{n\right\}_{n\in\NN_0}$. 

The $m$-function for any self-adjoint extension can also be written down in a standard way. Let $\tau\in\RR\cup\{\infty\}$ and $\cL_\tau$ refer to the self-adjoint operator acting via $\cL^{I,\al}_m$ on 
\begin{align*}
    \dom(\cL_\tau)&:=\left\{f\in\cD_{\ti{max}}\big|f\in\ker\left(\Gamma_1-\tau\Gamma_0\right)\right\}=\left\{f\in\cD_{\ti{max}}\big|f^{[1]}(0)=\tau f^{[0]}(0)\right\}.
\end{align*}
The corresponding Weyl $m$-function for $\la\in\rho(\cL_\infty)\cup\rho(\cL_\tau)$, as in Eq.~\eqref{e-generalbt}, is
\begin{align*}
M_\tau(\la)&=\frac{1+\tau M_\infty(\la)}{\tau-M_\infty(\la)} \\
&=\frac{(\la+m+\al)\Gamma(-\la)\Gamma(1-\al)-\tau\left[L_m^{\al-1}(0)\right]^2\Gamma(\al)\Gamma(1-\al-\la)}{\tau(\la+m+\al)\Gamma(-\la)\Gamma(1-\al)+\left[L_m^{\al-1}(0)\right]^2\Gamma(\al)\Gamma(1-\al-\la)}.
\end{align*}
Hence, if $\tau\notin\{0\}\cup\{\infty\}$, $\cL_\tau$ will have eigenvalues precisely when 
\begin{align*}
    \tau(\la+m+\al)\Gamma(-\la)\Gamma(1-\al)=-\left[L_m^{\al-1}(0)\right]^2\Gamma(\al)\Gamma(1-\al-\la).
\end{align*}
Solving for $\tau$ shows that eigenvalues of $\cL_\tau$ are just the level curves where $M_\infty(\la)=\tau$. Our example is now complete.

There are two main obstacles to generalizing the example that the reader should keep in mind. The first is the manipulation of the general solutions to the XOP expression so that they are in a practical format.  This step is not necessary in the above example as all calculations are explicit. The second is the determination of the normalizations for initial conditions; this is found by a simple calculation in the example and stated in Eq.~\eqref{e-iclaguerre}. In the general case, these obstacles are the main subject of Sections \ref{s-manipulations} and \ref{s-normalizations}, respectively.

\section{Exceptional Laguerre Operators}\label{s-mayaforxop}

We begin by discussing the behavior of solutions near the singular endpoints $x=0$ and $x=\infty$ for the XOP Laguerre expression found in Eq.~\eqref{e-XOPop}.  
\begin{theo}\label{t-frob}
For $\al>0$, let $\cL^\al$ be the XOP differential expression on the interval $(0,\infty)$.
\begin{enumerate}
    \item[(a.)] $\cL^\al$ is in the limit-circle case at $x=0$ for $-1<\al<1$ and is in the limit-point case at $x=0$ when $\alpha\geq 1$.
    \item[(b.)] $\cL^\al$ is in the limit-point case at $x=\infty$ for any $\al>0$.
\end{enumerate}
\end{theo}

\begin{proof}
    The point $x=0$ is, in the sense of Frobenius, a regular singular endpoint of the XOP expression $(\cL^\al-\lambda) f=0$ for any value $\lambda\in \mathbb C$. Using the classical Laguerre expression Eq.~\eqref{e-laguerre} to derive $\cL^\al$ associated with Eq.~\eqref{e-XOPop} yields an indicial equation of \[r(r+\al)=0.\] This indicial equation holds for any admissible choice of $\phi$ and $b$. Meanwhile, the point $x=\infty$ is an irregular singular point and a reduction of order method must be used to find two linearly independent solutions.  The result may be shown using the techniques within the proof of \cite[Theorem 3.3]{LLMS}.
\end{proof}

Since $\cL^\al$ is in the limit-circle case at $x=0$ for $-1<\al<1$, Glazman--Krein--Naimark theory requires that one appropriate boundary condition be imposed in order to generate a self-adjoint extension of the minimal operator. Thus if $-1<\al<1$, the deficiency index of $\cL$ is $(1,1)$. Meanwhile if $\al\geq 1$, the deficiency index is $(0,0)$.

\begin{ass}\label{a-alpha}
The XOP Laguerre expression $\cL^\al$ is assumed to have parameter $\al>-1$.
\end{ass}

In order to determine the spectrum of an XOP operator $\cL$, it is first necessary to determine the general solutions to the eigenvalue problem $(\cL-\la)f=0$. These solutions can be found through the intertwining property of Lemma \ref{l-intertwine}. By \cite[Section 3]{GGM2} the solutions may written as the Wronskian of a sequence of seed functions along with a solution to the eigenvalue problem of the COP operator.

Seed functions are the quasi-rational eigenfunctions of the Laguerre COP differential expression in Eqs.~\eqref{e-phi1}--\eqref{e-phi4} and will be indexed via
\begin{align}
f_j(x)&=L_{n_j}^{\al}(x),  &&j=1,\dots,r_1, \label{e-eigen1} \\
f_{r_1+j}(x)&=e^xL_{m_j}^{\al}(-x),  &&j=1,\dots,r_2, \label{e-eigen2} \\
f_{r_1+r_2+j}(x)&=x^{-\al}L^{-\al}_{m_j'}(x),  &&j=1,\dots,r_3, \label{e-eigen3} \\
f_{r_1+r_2+r_3+j}(x)&=e^x x^{-\al}L^{-\al}_{n_j'}(-x),  &&j=1,\dots,r_4, \label{e-eigen4}
\end{align}
with $r_1+r_2+r_3+r_4=r$, $n_1>\dots>n_{r_1}\geq0$, $m_1>\dots>m_{r_2}\geq0$, $n_1'>\dots>n_{r_1}'\geq0$ and $m_1'>\dots>m_{r_4}'\geq0$. The degrees of seed functions appearing in the Wronskian can be better described via two Maya diagrams, see Subsection \ref{ss-diagrams}. Denote these Maya diagrams by
\begin{equation}\label{e-generalizedmaya1}
M_1=\left(n_1',\dots,n_{r_4}'|n_1,\dots,n_{r_1}\right),
\end{equation}
and
\begin{equation}\label{e-generalizedmaya2}
M_2=\left(m_1',\dots,m_{r_3}'|m_1,\dots,m_{r_2}\right).
\end{equation}

The two general solutions of the classical Laguerre differential expression $\ell^{\al}$ given by Eq.~\eqref{e-laguerre} are denoted
\begin{equation}\label{e-shortsolns1}
    h^{\al}(x,\la):=M(-\la,\al+1,x)= \,_1F_1(-\la,\al+1,x)\hspace{2em} \text{for }\al+1\notin-\NN_0,
\end{equation}
and
\begin{equation}\label{e-shortsolns2}
    \wt{h}^{\al}(x,\la):=x^{-\al}M(-\la-\al,1-\al,x)\hspace{2em} \text{for }\al+1\notin\NN>1.
\end{equation}

A generalized Laguerre polynomial, $\Omega\ci{M_1,M_2}^{\al}$, is a Wronskian of seed functions with the same parameter $\al$ and distinct degrees. A prefactor is required to make it a polynomial. Let $\Omega\ci{M_1,M_2}^{\al}$ be defined by
\begin{equation}\label{e-particularlaguerre}
\Omega\ci{M_1,M_2}^{\al}(x)=e^{-(r_2+r_4)x}x^{(\al+r_1+r_2)(r_3+r_4)}\cdot\Wr\left[f_1,\dots,f_r\right],
\end{equation}
where $f_1,\dots,f_r$ are as in Eqs.~\eqref{e-eigen1}-\eqref{e-eigen4}. If both Maya diagrams are trivial or the associated partitions are empty, the generalized Laguerre polynomial is a constant function.

Not every set of seed functions will produce a valid XOP expression. The following notion will determine admissibility; ensuring that the generalized Laguerre polynomial defines a $\phi$ as in Subsection \ref{s-prelims}. The Maya diagrams $M_1$ and $M_2$ of Eqs.~\eqref{e-generalizedmaya1} and \eqref{e-generalizedmaya2} can be presented as partitions, denoted respectively by $\mu$ and $\nu$ in canonical position or $\mu'$ and $\nu'$ in conjugate canonical position. The lengths of these partitions, denoted $r(\mu)$ and $r(\nu)$ or $r(\mu')$ and $r(\nu')$, are determined by using the values $t_1$, $t_2$ and $t_1'$, $t_2'$ that needed to shift $M_1$ and $M_2$ into their canonical and conjugate canonical positions. In particular, 
\begin{equation}\label{e-rs1}
\begin{aligned}
r(\mu)&=
\begin{cases}
r_1+n_1'+1-r_4 & t_1>0 \\
r_1 & t_1=0 \\
r_1-\min\left\{k\in\NN_0:k\notin\{n_i\}_{i=1}^{r_1}\right\} & t_1<0
\end{cases}, \end{aligned}\end{equation}
and
\begin{equation}\label{e-rs2}
\begin{aligned}
r(\nu)&=
\begin{cases}
r_2+m_1'+1-r_3 & t_2>0 \\
r_2 & t_2=0 \\
r_2-\min\left\{k\in\NN_0:k\notin\{m_i\}_{i=1}^{r_2}\right\} & t_2<0
\end{cases} .
\end{aligned}
\end{equation}
The lengths of the conjugate partitions, $r(\mu')$ and $r(\nu')$, can be determined similarly but will not play an important role in our calculations. 

In canonical position, the Maya diagram can thus be relabeled and written as a partition defined by 
\begin{equation}\label{e-canonicalpartition1}
    \begin{aligned}
\mu&=\left(\mu_1,\dots,\mu_{r(\mu)}\right), \hspace{2em} \mu_j=n_j-r(\mu)+j, \hspace{2em} &&j=1,\dots,r(\mu),
 \end{aligned}
\end{equation}
and
\begin{equation}\label{e-canonicalpartition2}
    \begin{aligned}
\nu&=\left(\nu_1,\dots,\nu_{r(\nu)}\right), \hspace{2em} \nu_j=m_j-r(\nu)+j, \hspace{2em} &&j=1,\dots,r(\nu).
    \end{aligned}
\end{equation} where $r(\mu)+r(\nu)=r$.
Conjugate partitions are relabeled so that $r(\mu')+r(\nu')=r'$ and defined analogously. Partitions also define Young diagrams, where the number of boxes in each row corresponds to the entry of the partition. Conjugate partitions are then obtained by reflection over the main diagonal of the corresponding Young diagram. 
\begin{defn}
A partition $\upsilon=(\upsilon_1,\dots,\upsilon_r)$ with $\upsilon_r\geq 1$ is \textit{even} if $r$ is even and $\upsilon_{2j-1}=\upsilon_{2j}$ for every $j=1,\dots,r/2$.
\end{defn}
In order for a Maya diagram to represent a valid Laugerre XOP expression, we make the following assumption:
\begin{ass}\label{a-partitions}
The partition $\mu$ is even. 
\end{ass}

It follows from \cite{D, DP} that the polynomial $\Omega^{\al}_{\mu,\nu}$ has no zeros on $[0,\infty)$ if and only if $\mu$ is an even partition.  It is for this reason that we require an even partition for admissibility. For the remainder of the manuscript, it is assumed that all Maya diagrams will satisfy Assumption \ref{a-partitions}. Convention states that the empty partition (when $r=0$) is even. 

In order to apply the Darboux transform to the COP solutions of Eqs.~\eqref{e-shortsolns1} and \eqref{e-shortsolns2}, it is necessary to add another entry to the corresponding Wronskian and adjust the prefactor. In particular, denote
\begin{align}\label{e-firstgen} 
\Omega\ci{M_1,M_2}\left[h^{\al}(x,\la)\right]&=e^{-(r_2+r_4)x}x^{(\al+r_1+r_2+1)(r_3+r_4)}\cdot\Wr\left[f_1,\dots,f_r,h\right],
\end{align}
and
\begin{align}\label{e-secondgen}
\Omega\ci{M_1,M_2}\left[\wt{h}^{\al}(x,\la)\right]&=e^{-(r_2+r_4)x}x^{(\al+r_1+r_2)(r_3+r_4+1)}\cdot\Wr\left[f_1,\dots,f_r,\wt{h}\right].
\end{align}
For the sake of convenience, we will often refer to Eq.~\eqref{e-firstgen} as a solution of the first kind and Eq.~\eqref{e-secondgen} as a solution of the second kind. The parameter $\al$ is written only for the solution but it should be understood that this is also the parameter for all seed functions in the Wronskian.

\begin{remme}\label{r-poly} It is clear that adding a column with a solution $h^\al(x,\la)$ or $\wt{h}^\al(x,\la)$ to the Wronskian changes the expression from a polynomial to an infinite series. However, consider the case $\la=n$ in Eq.~\eqref{e-firstgen} where $n\in\NN_{\mu,\nu}$, the natural numbers without a subset that is determined by $\mu$ and $\nu$, see \cite[Definition 2.9]{BK} for details. Then the series terminates, and if multiplied by a normalizing constant will yield an exceptional Laguerre polynomial $L^{\al}_{\mu,\nu,n}$. These polynomials can emerge from either solution, and will be discussed further in Assumption \ref{a-type3} and Remark \ref{r-type3}.
\end{remme}

Under our standard assumptions, these Laguerre XOPs form a complete set of orthogonal polynomials on the positive real line \cite{D, DP}.

\begin{lem}\label{l-ortho}
Suppose $\al+r>-1$ and $\mu$ is an even partition. Then the polynomials $L^{\al}_{\mu,\nu,n}$ for $n\in\NN_{\mu,\nu}$ are orthogonal on $[0,\infty)$ with respect to the positive weight function
\begin{align}\label{e-newweight}
    W^{\al}_{\mu,\nu}(x)=\frac{x^{\al+r}e^{-x}}{\left[\Omega_{\mu,\nu}^{\al}(x)\right]^2},\hspace{2em}x>0.
\end{align}
That is, if $n,m\in\NN_{\mu,\nu}$ with $n\neq m$, then
\begin{align*}
    \int_0^{\infty}L^{\al}_{\mu,\nu,n}(x)L^{\al}_{\mu,\nu,m}(x)W^{\al}_{\mu,\nu}(x)dx=0.
\end{align*}
Moreover, they form a complete orthogonal set in $L^2\left([0,\infty),W^{\al}_{\mu,\nu}(x)dx\right)$.
\end{lem}

Note that the assumptions of Lemma \ref{l-ortho} are not the best possible conditions, but they are convenient and straightforward to apply. More specific conditions for admissibility can be found in \cite{D, DP}.

Also, the solution of the second kind defined by Eq.~\eqref{e-secondgen} does not currently possess the correct asymptotic behavior near $0$; it should have asymptotic behavior of order $x^{-\al-r}$ to fit the Frobenius analysis of the corresponding exceptional Laguerre expression in Theorem \ref{t-frob}. The current formulation of this solution of the second kind as a polynomial is simply more convenient to work with and avoids further complicating notation. This discrepancy will be rectified in Section \ref{s-normalizations}, where other behavior near $x=0$ is needed to ensure the solution is in the proper Hilbert space. 

Finally, the construction of a general exceptional Laguerre expression requires beginning with the classical Laguerre expression, which has a parameter $\al$, and applying Darboux transforms.  In this process, $\al$ may be shifted. Therefore, the parameter $\al$ in Lemma \ref{l-ortho} should not be thought of as simply the same $\al$ from the original COP expression. One of the main goals of Section \ref{s-manipulations} is to track the necessary changes to the classical $\al$ as operations are performed.

\section{Manipulation of Maya Diagrams}\label{s-manipulations}

Information must be extracted from the two solutions of the XOP Laguerre expression given in Eqs.~\eqref{e-firstgen} and \eqref{e-secondgen} in order to build the Weyl $m$-function.  The results follow the general methods of \cite[Theorem 4.2, Lemma 4.4]{BK}, which may be consulted for additional insight. The presentation here involves significant additions. Most importantly, it is necessary to track changes of the spectral parameter through the shifts as well as any constants produced by the spectral parameter. The format of these solutions can vary significantly based on how the Maya diagrams are shifted. In this section, the Maya diagrams are translated into standard positions where the necessary information can be extracted. Solutions of the first kind have their Maya diagrams shifted into canonical position, while solutions of the second kind are shifted into conjugate canonical position. A few additional assumptions are necessary.

\begin{ass}\label{a-mayaorder}
	Without loss of generality, we assume that $M_1$ is shifted to canonical form prior to any shifts that may be necessary for $M_2$. Indeed, the steps of the shifting process outlined in Theorems \ref{t-new4.2} and \ref{t-new4.22} may be applied to $M_1$ and $M_2$ in a variety of orders, but this assumption minimizes notational complexity.
\end{ass}

\begin{ass}\label{a-signs}
Factors of $(-1)^d$ for some $d\in\NN$ that arise in calculations are omitted throughout the remainder of the manuscript. Tracking this factor requires substantial notation and does not meaningfully contribute to results within. Such a factor appears in the equations of Theorems \ref{t-new4.2}, \ref{t-new4.22}, \ref{t-first0}, Corollaries \ref{c-simple4.2}, \ref{c-second0}, \ref{c-simple0} and the $m$-functions of Section \ref{s-mfunctions}. Therefore, all equalities in these results are correct up to a change of sign. The cumulative effect of these factors throughout the construction of the Weyl $m$-function can still be determined by enforcing the property that this function is a so-called Nevanlinna--Herglotz function: an analytic self-map of the upper half-plane. See i.e.~\cite[Section 6]{GLN} for examples where this property is utilized in a similar context. 
\end{ass}

We begin by shifting the Maya diagrams defining a solution of the first kind into canonical position. The notation $x^{(n)}=x(x+1)\cdots(x+n-1)$ and $x_{(n)}=x(x-1)\cdots(x-n+1)$ is used to denote the rising and falling factorials, respectively.

	\begin{theo}\label{t-new4.2}
		Let $f_1,\dots,f_r$ be as in Eqs.~\eqref{e-eigen1}-\eqref{e-eigen4}, $M_1$ and $M_2$ be as in Eqs.~\eqref{e-generalizedmaya1} and \eqref{e-generalizedmaya2}, and $\mu$ and $\nu$ be their associated partitions after shifts $t_1$ and $t_2$ to canonical position, respectively. Then for $\la\neq n_i$ with $i=1,\dots,r_1$,
		\begin{align*}
			\Omega\ci{M_1,M_2}\left[h^{\al}(x,\la)\right]=C_1(\al,\la,M_1)C_2(\al,\la,M_2)C_3(\al,M_1,M_2)\Omega_{\mu,\nu}\left[h^{\al-t_1-t_2}(x,\la+t_1)\right],
		\end{align*}
		where
		\begin{align}
			C_1(\al,\la,M_1)&=\begin{cases}
				\frac{(-\la)^{(|t_1|)}}{(\al+1)^{(|t_1|)}} & \text{ for } t_1<0, \\
				1 & \text{ for }t_1=0, \\
				\frac{(\al)_{(t_1)}}{\prod_{k_1}(-\la-k_1-1)} & \text{ for }t_1>0, \\
			\end{cases} \label{e-c1}
		\end{align}
		for $k_1\in\{0,\dots,t_1-1\}$ such that $k_1\notin\{n_j'\}_{j=1}^{r_4}$;
	
		\begin{align}
			C_2(\al,\la,M_2)&=\begin{cases}
				\frac{(\al+1+\la)^{(|t_2|)}}{(\la+t_1)^{|t_2|}} & \text{ for }t_2<0, \\
				1 & \text{ for }t_2=0, \\
				\frac{(-\la-t_1)^{t_2-r_3}\cdot\prod_{i=1}^{r_3}(\al-t_1-m_i')}{\prod_{k_2}(\la+\al-k_2)} & \text{ for }t_2>0, \\
			\end{cases} \label{e-c2}
		\end{align}
		for $k_2\in\{0,\dots,t_2-1\}$ such that $k_2\notin\{m_j'\}_{j=1}^{r_3}$; and 
		
		\begin{align}\label{e-c3}
			C_3(\al,M_1,M_2)=&\prod_{j=1}^{r_1}\prod_{k=1}^{r_3}\left(m_k'-\al-n_j\right)\prod_{j=1}^{r_2}\prod_{k=1}^{r_4}\left(n_k'-\al-m_j\right) \\ \nonumber
			&\times\prod_{j=1}^{r_1}\prod_{k=1}^{r_4}\left(n_j+n_k'+1\right)\prod_{j=1}^{r_2}\prod_{k=1}^{r_3}\left(m_j+m_k'+1\right). 
		\end{align}
	
	\end{theo}


The constants from Theorem \ref{t-new4.2} will be collectively denoted as 
\begin{align}\label{e-C}
C:=C_1(\al,\la,M_1)C_2(\al,\la,M_2)C_3(\al,M_1,M_2).
\end{align}

The proof of Theorem \ref{t-new4.2} requires several intermediate steps. Recall the following derivative identities:
\begin{equation}\label{e-derivatives}
\begin{aligned}
\frac{d}{dx}\left(L_n^{\al}(x)\right)&=-L_{n-1}^{\al+1}(x), \\
\frac{d}{dx}\left(e^x L_n^{\al}(-x)\right)&=e^x L_{n}^{\al+1}(-x), \\
\frac{d}{dx}\left(x^{-\al}L_n^{-\al}(x)\right)&=(n-\al)x^{-\al-1}L_{n}^{-\al-1}(x), \\
\frac{d}{dx}\left(x^{-\al} e^x L_n^{-\al}(-x)\right)&=(n+1)x^{-\al-1}e^x L_{n+1}^{-\al-1}(-x).
\end{aligned}
\end{equation}
The derivatives of Eq.~\eqref{e-derivatives} involve normalizing constants, which result in significant cancellations, but otherwise are special cases of the following derivative identities for confluent hypergeometric functions
\begin{equation}\label{e-solnderivs}
\begin{aligned}
\frac{d}{dx}\left(M(a,b,x)\right)&=\frac{a}{b}M(a+1,b+1,x), \\
\frac{d}{dx}\left(e^x M(a,b,-x)\right)&=\frac{b-a}{b}e^x M(a,b+1,-x), \\
\frac{d}{dx}\left(x^{b-1}M(a,b,x)\right)&=(b-1)x^{b-2}M(a,b-1,x), \\
\frac{d}{dx}\left(x^{b-1} e^x M(a,b,-x)\right)&=-(b-1)x^{b-2}e^{x}M(a-1,b-1,-x).
\end{aligned}
\end{equation}
The results of Eq.~\eqref{e-solnderivs} will be used to take derivatives of $h$ and $\wt{h}$. 

Additionally, the following elementary Wronskian identities that will be useful for calculations. Assume that the functions $f_1,\dots,f_r,g,$ and $h$ are all sufficiently differentiable. Then
\begin{align}
    \Wr\left[g\cdot f_1,\dots,g\cdot f_r\right]&=\left(g(x)\right)^r\cdot\Wr\left[f_1,\dots,f_r\right], \label{e-wronskian1} \\
    \Wr\left[f_1\circ h,\dots,f_r\circ h\right]&=\left(h'(x)\right)^{\frac{r(r-1)}{2}}\cdot\Wr\left[f_1,\dots,f_r\right]\cdot h(x).\label{e-wronskian2}
\end{align}

To prove Theorem \ref{t-new4.2}, we begin with a Corollary that involves simplifications when $0$ appears in the encoding of one of the Maya diagrams $M_1$ or $M_2$. This is an alteration of \cite[Lemma 4.4]{BK}. 

\begin{cor}\label{c-0reduction}
Let $M_1$ and $M_2$ be given by Eqs.~\eqref{e-generalizedmaya1} and \eqref{e-generalizedmaya2}.
\begin{itemize}
    \item[(a)] If $n_{r_1}=0$, then
    \begin{align}\label{e-r10red}
    \Omega\ci{M_1,M_2}\left[h^{\al}(x,\la)\right]=\left(\frac{-\la}{\al+1}\prod_{i=1}^{r_3}(m_i'-\al)\prod_{i=1}^{r_4}(n_i'+1)\right)\Omega\ci{M_1-1,M_2}\left[h^{\al+1}(x,\la-1)\right].
    \end{align}
    
    \item[(b)] If $n_{r_4}'=0$, then
    \begin{align}\label{e-r40red}
    \Omega\ci{M_1,M_2}\left[h^{\al}(x,\la)\right]=\left(\al\prod_{i=1}^{r_1}(n_i+1)\prod_{i=1}^{r_2}(m_i+\al)\right) \Omega\ci{M_1+1,M_2}\left[h^{\al-1}(x,\la+1)\right].
    \end{align}
    
    \item[(c)] If $m_{r_2}=0$, then 
    \begin{align}\label{e-r20red}
    \Omega\ci{M_1,M_2}\left[h^{\al}(x,\la)\right]=\left(\frac{\al+1+\la}{-\la}\prod_{i=1}^{r_3}(m_i'+1)\prod_{i=1}^{r_4}(n_i'-\al)\right) \Omega\ci{M_1,M_2-1}\left[h^{\al+1}(x,\la)\right].
    \end{align}
    
    \item[(d)] If $m_{r_3}'=0$, then 
    \begin{align}\label{e-r30red}
    \Omega\ci{M_1,M_2}\left[h^{\al}(x,\la)\right]=\left(\al\prod_{i=1}^{r_1}(n_i+\al)\prod_{i=1}^{r_2}(m_i+1)\right)\Omega\ci{M_1,M_2+1}\left[h^{\al-1}(x,\la)\right].
    \end{align}
\end{itemize}
\end{cor}

\begin{proof}
We give only the proof of (b) as the other claims follow from using the same procedure. Let $n_{r_4}'=0$. Then $f_r=e^x x^{-\al}$. Removing a factor of $e^x x^{-\al}$ from each column via Eq.~\eqref{e-wronskian1} and expanding about the second to last column yields 
\begin{align*}
    \Omega\ci{M_1,M_2}\left[h^{\al}(x,\la)\right]&=\left(e^x x^{-\al}\right)^{(r+1)}e^{-(r_2+r_4)x}x^{(\al+r_1+r_2+1)(r_3+r_4)} \\
    &\hspace{3em}\times\Wr\left[e^{-x}x^{\al}f_1,\dots,e^{-x}x^{\al}f_{r-1},1,e^{-x}x^{\al}h\right](x,\la) \\
    &=(-1)^{r-1}\left(e^x x^{-\al}\right)^{(r+1)}e^{-(r_2+r_4)x}x^{(\al+r_1+r_2+1)(r_3+r_4)} \\
    &\hspace{3em}\times\Wr\left[\frac{d}{dx}\left(e^{-x}x^{\al}f_1\right),\dots,\frac{d}{dx}\left(e^{-x}x^{\al}f_{r-1}\right),\frac{d}{dx}\left(e^{-x}x^{\al}h\right)\right](x,\la).
\end{align*}
Let $\wt{M}_1=M_1+1$ and let $\wt{f}_1,\dots,\wt{f}_{\wt{r}}$ be the functions associated with the Maya diagram $\wt{M}_1$ and parameter $\wt{\al}=\al-1$. Also let $\wt{r}_1=r_1$, $\wt{r}_2=r_2$, $\wt{r}_3=r_3$, $\wt{r}_4=r_4-1$ and $\wt{r}=r-1$. Eqs.~\eqref{e-derivatives} and \eqref{e-solnderivs} imply that
\begin{align*}
    \frac{d}{dx}\left(e^{-x}x^{\al}f_j\right)&=c_j e^{-x}x^{\al-1}\wt{f}_j \hspace{3em} \text{ for }j=1,\dots,r-1,
\end{align*}
and
\begin{align*}
    \frac{d}{dx}\left(e^{-x} x^{\al} M(-\la,\al+1,x)\right)&=\al e^{-x}x^{\al-1}M(-\la-1,\al,x)\\
    &=\al e^{-x}x^{\al-1}h^{\al-1}(x,\la+1),
\end{align*}
where $c_j=n_j+1$ for $j=1,\dots,r_1$; $c_{r_1+j}=m_j+\al$ for $j=1,\dots,r_2$; $c_{r_1+r_2+j}=-1$ for $j=1,\dots,r_3$; and $c_{r_1+r_2+r_3+j}=1$ for $j=1,\dots,r_4-1$. Each column of the Wronskian has a corresponding constant removed, and then the common factor $e^{-x}x^{\al-1}$ is removed by Eq.~\eqref{e-wronskian1}, yielding
\begin{align*}
    \Omega\ci{M_1,M_2}\left[h^{\al}(x,\la)\right]&=\wt{c}_1\left(e^x x^{-\al}\right)^{(r+1)}\left(e^{-x} x^{\al-1}\right)^{(r)}e^{-(r_2+r_4)x}x^{(\al+r_1+r_2+1)(r_3+r_4)} \\
    &\hspace{5em}\times\Wr[\wt{f}_1,\dots,\wt{f}_{r-1},h](x,\la+1) \\
    &=\wt{c}_1e^{-(r_2+r_4-1)x}x^{(\al+r_1+r_2+1)(r_3+r_4)-\al-r}\cdot\Wr[\wt{f}_1,\dots,\wt{f}_{r-1},h](x,\la+1),
\end{align*}
where 
\begin{align*}
\wt{c}_1=\al(-1)^{r-1}\prod_{j=1}^{r_4-1}c_j=\al(-1)^{r_1+r_2+r_4-1}\prod_{j=1}^{r_1}(n_j+1)\prod_{j=1}^{r_2}(m_j+\al).
\end{align*}
Note the analogous step in the proof of \cite[Lemma 4.4]{BK} omits a $-\al$ term from the exponent of $x$ but the end result is unchanged. Claim (b) follows because
\begin{align*}
(\al+r_1+r_2+1)(r_3+r_4)-\al-r&=(\al+r_1+r_2)(r_3+r_4-1)\\&=(\wt{\al}+\wt{r}_1+\wt{r}_2+1)(\wt{r}_3+\wt{r}_4).
\end{align*}
\end{proof}

Collectively, (a)--(d) of Corollary \ref{c-0reduction} describe how the spectral parameter $\la$ and $\al$ change as well as the constants produced with any shift of $M_1$ or $M_2$ one unit left or right.  We will continue to keep track of these changes while iterating the shifts until the Maya diagrams are in canonical or conjugate canonical form. Note that we will appeal to \cite[Section 4.4]{BK} to show the exact form of the products in the statement of Theorem \ref{t-new4.2}, as there are no alterations to that argument, we simply keep track of additional parameters.

\begin{proof}[Proof of Theorem \ref{t-new4.2}]
	To begin, we verify Eqs.~\eqref{e-c1} and \eqref{e-c2}.  These results only consider contributions of the non-product coefficient factors found in Eqs.~\eqref{e-r10red}-\eqref{e-r30red}.  These product terms are collected and addressed in the verification of Eq.~\eqref{e-c3}. 
	
	We begin by shifting $M_1$ into canonical form.  If $t_1=0$, the Maya diagram $M_1$ is already in canonical form and no shift is required.  If $t_1<0$, then $|t_1|$ boxes to the immediate right of the origin are filled; that is, $n_{r_1-j}=j$ for $j=0,\dots,|t_1|-1$.  Hence Eq.~\eqref{e-r10red} will be applied $|t_1|$ times to $M_1$.  
	Therefore, $t_1$ applications of Eq.~\eqref{e-r10red} yield coefficient factor
	
	\begin{align*}
		\frac{(-\la)^{(|t_1|)}}{(\al+1)^{(|t_1|)}}.
	\end{align*}
	Note that for $t_1<0$, the parameters of $h$ are now $\al-t_1$ and $\la+t_1$.  
	
	If $t_1>0$, denote
	\begin{align*}
		\wt{M}_1=M_1+1:\hspace{1em}\left(\wt{n}_1',\dots,\wt{n}_{\wt{r}_4}'~|~\wt{n}_1,\dots,\wt{n}_{\wt{r}_1}\right).
	\end{align*}
	Determining the appropriate equation of Corollary \ref{c-0reduction} to apply depends upon whether the box immediately to the left of the origin of $M_1$ is or is not filled.  There are two cases:
	\begin{itemize}
		\item[(i)] If the box to the left of the origin is empty then $n_{r_4}'=0$, $\wt{r}_4=r_4-1$, $\wt{r}_1=r_1$, and
		\begin{align*}
			\wt{M}_1:\hspace{1em}\left(n_1'-1,\dots,n_{r_4-1}'-1~|~n_1+1,\dots,n_{r_1}+1\right).
		\end{align*} 
		\item[(ii)] If the box to the left of the origin is filled then $\wt{n}_{r_1+1}=0$, $\wt{r}_4=r_4$, $\wt{r}_1=r_1+1$, and 
		\begin{align*}
			\wt{M}_1:\hspace{1em}\left(n_1'-1,\dots,n_{r_4-1}'-1~|~n_1+1,\dots,n_{r_1}+1,0\right).
		\end{align*}
	\end{itemize} In total, $M_1$ requires $t_1$ shifts in order to be in canonical form.  There will be $r_4$ boxes to the left of the origin that are empty; hence we are in Case (i) of above and Eq.~\eqref{e-r40red} will be applied $r_4$ times.  As a result, each application of Eq.~\eqref{e-r40red}  lowers $\al$ by one and increases $\la$ by one.  Additionally for each $i\in \{1,\ldots r_4\}$, a factor of $(\al-n_i')$ will be contributed to the coefficient. Case (ii) occurs $t_1-r_4$ times and Eq.~\eqref{e-r10red} will be applied.  As $M_1$ needs to be shifted to the right, Eq.~\eqref{e-r10red} must be restated as

	\begin{align*}
		\Omega\ci{M_1,M_2}\left[h^{\al}(x,\la)\right]=\left[\left(\frac{-\la-1}{\al}\right)\prod_{i=1}^{r_3}(m_i'-\al+1)\prod_{i=1}^{r_4}n_i' \right]^{-1} \Omega\ci{M_1+1,M_2}\left[h^{\al-1}(x,\la+1)\right].
	\end{align*}
	Each application of Eq.~\eqref{e-r10red} also lowers $\al$ by one and increases $\la$ by one.  A factor of $(\al-k_1)/(-\la-k_1-1)$ for each $k_1\in\{0,\dots,t_1-1\}$ such that $k_1\notin\{n_j'\}_{j=1}^{r_4}$ will be collected by the coefficient.
	
	Together contributions from all Case (i) and Case (ii) shifts produce a coefficient of
	\begin{align}\label{e-type1shift}
		\frac{(\al)_{(t_1)}}{\prod_{k_1}(-\la-k_1-1)},
	\end{align} and the parameters of $h$ are now $\al-t_1$ and $\la+t_1$.
	Eq.~\eqref{e-c1} has been verified. 
	
	Now, we consider the adjustments required to shift $M_2$ into canonical form.  As $M_1$ is first put into canonical form, the parameters for $h$ are now $\al-t_1$ and $\la+t_1$ for  $t_1\in \mathbb Z$. If $t_2=0$, the Maya diagram $M_2$ is already in canonical form and no shift is required.  If $t_2<0$,  then $|t_2|$ boxes to the immediate right of the origin are filled; that is, $m_{r_2-j}=j$ for $j=0,\dots,|t_2|-1$.  Hence Eq.~\eqref{e-r20red} will be applied $|t_2|$ times to $M_2$.  Therefore,  $t_2$ applications of Eq.~\eqref{e-r20red} yields additional coefficient factors
	\begin{align*}
	\frac{(\al+1+\la)^{(|t_2|)}}{(\la+t_1)^{|t_2|}}.
	\end{align*}
	
	For $t_2>0$, verification follows in a similar fashion as the $M_1$ case except it relies on the reductions, and therefore the constants from, Eq.~\eqref{e-r20red} and the restatement of Eq.~\eqref{e-r20red}: 
	\begin{align*}
		\Omega\ci{M_1,M_2}\left[h^{\al}(x,\la)\right]=\left[\left(\frac{\al+\la}{-\la}\right)\prod_{i=1}^{r_3}m_i'\prod_{i=1}^{r_4}(n_i'-\al+1)\right]^{-1}\Omega\ci{M_1,M_2+1}\left[h^{\al-1}(x,\la)\right].
	\end{align*}

	The contribution of the products found in Eqs.~\eqref{e-r10red}-\eqref{e-r30red} is contained in $C_3$ and the proof is identical to \cite[Section 4.4]{BK}. 
	\end{proof}

The method of proof of Theorem \ref{t-new4.2} may be repeated with a solution of the second kind, but there are a few new obstacles. First, it is necessary to state the result with conjugate partitions; shifting the Maya diagrams to the conjugate canonical position so that the partitions are determined by the eigenfunctions in Eqs.~\eqref{e-eigen3} and \eqref{e-eigen4}. This is due to the fact that the solution $\wt{h}$ has the same format and similar derivative formula as the seed function Eq.~\eqref{e-eigen3}---a shift to canonical position would not change this fact. The conjugate canonical position, on the other hand, still allows the use of partitions.  Second, there is no formula giving the product of constants from the shifting functions (stated as $C_3$ in Theorem \ref{t-new4.2}) in the literature. It is therefore necessary to prove such a formula. 
 
\begin{theo}\label{t-new4.22}
	Let $f_1,\dots,f_r$ be as in Eqs.~\eqref{e-eigen1}-\eqref{e-eigen4}, $M_1$ and $M_2$ be as in Eq.~\eqref{e-generalizedmaya1} and \eqref{e-generalizedmaya2}, and $\mu',\nu'$ be their associated conjugate partitions after shifts $t_1'$ and $t_2'$ to conjugate canonical position, respectively. Then for $\la\neq m_i'$ with $i=1,\dots,r_3$,
	\begin{align*}
		\Omega\ci{M_1,M_2}\left[\wt{h}^{\al}(x,\la)\right]&=D_1(\al,\la,M_1)D_2(\al,\la,M_2)D_3(\al,M_1,M_2) \Omega_{\mu',\nu'}\left[\wt{h}^{\al-t_1'-t_2'}(x,\la+t_1')\right],
	\end{align*}
	where
	\begin{align}
		D_1(\al,\la,M_1)&=
		\begin{cases}
			\frac{(-\al)_{(|t_1'|)}}{\prod_{k_1'} (\la-k_1')} & \text{ for }t_1'<0, \\
			1 & \text{ for }t_1'=0, \\
						\frac{(\la+1)^{(t_1')}}{(1-\al)^{(t_1')}} & \text{ for }t_1'>0, 
		\end{cases} \label{e-d1}
	\end{align}
	for $k_1'\in\{0,\dots,|t_1'|-1\}$ such that $k_1'\notin\{n_j\}_{j=1}^{r_4}$;
	\begin{align}
		D_2(\al,\la,M_2)&=\begin{cases}
			\frac{(-\al+t_1')_{(|t_2'|)}}{\prod_{k_2'}(-\la-\al-1-k_2')} & \text{ for }t_2'<0, \\
			1 & \text{ for }t_2'=0, \\
			\frac{(-\la-\al)^{(t_2')}}{(1-\al+t_1')^{(t_2')}} & \text{ for }t_2'>0, 
		\end{cases}\label{e-d2}
	\end{align}
	for $k_2'\in\{0,\dots,|t_2'|-1\}$ such that $k_2'\notin\{m_j\}_{j=1}^{r_3}$.
	
	\begin{equation}\label{e-d3}
		\begin{aligned}
			D_3(\al,M_1,M_2)=&\prod_{j=1}^{r_1}\prod_{k=1}^{r_3}\left(n_j+\al-m_k'\right)\prod_{j=1}^{r_2}\prod_{k=1}^{r_4}\left(m_j+\al-n_k'\right) \\
			&\times\prod_{j=1}^{r_1}\prod_{k=1}^{r_4}\left(n_j+n_k'+1\right)\prod_{j=1}^{r_2}\prod_{k=1}^{r_3}\left(m_j+m_k'+1\right).
		\end{aligned} 
	\end{equation}
\end{theo}


The constants from Theorem \ref{t-new4.22} are collectively denoted as 
\begin{align}\label{e-D}
D:=D_1(\al,\la,M_1)D_2(\al,\la,M_2)D_3(\al,M_1,M_2),
\end{align}
and will be used in Section \ref{s-mfunctions}.

The change in parameters and constants is now tracked for a shift of one unit in either direction for the Maya diagrams. These are similar to those in Corollary \ref{c-0reduction} but with extra contributions emerging from $\wt{h}$.

\begin{cor}\label{c-0reduction2}
Let $M_1$ and $M_2$ be given by Eqs.~\eqref{e-generalizedmaya1} and \eqref{e-generalizedmaya2}.
\begin{itemize}
    \item[(a)] If $n_{r_1}=0$, then
    \begin{align}\label{e-r10red2}
    \Omega\ci{M_1,M_2}\left[\wt{h}^{\al}(x,\la)\right]=\left((-\al)\prod_{i=1}^{r_3}(m_i'-\al)\prod_{i=1}^{r_4}(n_i'+1)\right)\Omega\ci{M_1-1,M_2}\left[\wt{h}^{\al+1}(x,\la-1)\right].
    \end{align}
    
    \item[(b)] If $n_{r_4}'=0$, then
    \begin{align}\label{e-r40red2}
    \Omega\ci{M_1,M_2}\left[\wt{h}^{\al}(x,\la)\right]=\left(\frac{\la+1}{1-\al}\prod_{i=1}^{r_1}(n_i+1)\prod_{i=1}^{r_2}(m_i+\al)\right)\Omega\ci{M_1+1,M_2}\left[\wt{h}^{\al-1}(x,\la+1)\right].
    \end{align}
    
    \item[(c)] If $m_{r_2}=0$, then
    \begin{align}\label{e-r20red2}
    \Omega\ci{M_1,M_2}\left[\wt{h}^{\al}(x,\la)\right]=\left((-\al)\prod_{i=1}^{r_3}(m_i'+1)\prod_{i=1}^{r_4}(n_i'-\al)\right)\Omega\ci{M_1,M_2-1}\left[\wt{h}^{\al+1}(x,\la)\right]. 
    \end{align}
    
    \item[(d)] If $m_{r_3}'=0$, then
    \begin{align}\label{e-r30red2}
    \Omega\ci{M_1,M_2}\left[\wt{h}^{\al}(x,\la)\right]=\left(\frac{-\la-\al}{1-\al}\prod_{i=1}^{r_1}(n_i+\al)\prod_{i=1}^{r_2}(m_i+1)\right)\Omega\ci{M_1,M_2+1}\left[\wt{h}^{\al-1}(x,\la)\right].
    \end{align}
\end{itemize}
\end{cor}

\begin{proof}
It is sufficient to use the derivative formulas in Eq.~\eqref{e-solnderivs} to calculate the changed parameters and constants in each of the situations. Otherwise, the proof is completely analogous to that of Corollary \ref{c-0reduction}.
\end{proof}

Verifying Eqs.~\eqref{e-d1} and \eqref{e-d2} follows the same procedure as the proof of Theorem \ref{t-new4.2}, but relies on the equations found within Corollary \ref{c-0reduction2}.  The method also differs slightly from that of Theorem \ref{t-new4.2} as the Maya diagrams are to be shifted into the conjugate canonical form. As no existing result addresses the formulation of Eq.~\eqref{e-d3}, the proof is included below.  The method closely follows the result of \cite[Section 4.4]{BK}.

\begin{proof}[Proof of Theorem \ref{t-new4.22}]
	We note that the shifting procedure for $t_1<0$ depends on the Maya diagram for $M_1$ and whether the box of immediately to the right of the origin is or is not filled.  This presents two cases:
	\begin{itemize}
		\item[(i)] If the box to the right of the origin is filled then $n_{r_1}=0$, $\wt{r}_1=r_1-1$, $\wt{r}_4=r_4$, and
		\begin{align}\label{e-firstcase}
			\wt{M}_1:\hspace{1em}\left(n_1'+1,\dots,n_{r_4}'+1~|~n_{r_1-1}-1,\dots,n_{r_1}-1\right).
		\end{align}
		\item[(ii)] If the box to the right of the origin is empty then $\wt{n}_{r_4+1}'=0$, $\wt{r}_4=r_4+1$, $\wt{r}_1=r_1$, and
		\begin{align}\label{e-secondcase}
			\wt{M}_1:\hspace{1em}\left(n_1'+1,\dots,n_{r_4}'+1,0~|~n_1-1,\dots,n_{r_1}-1\right).
		\end{align}
	\end{itemize}  
	In Case (ii), Eq.~\eqref{e-r40red2} is rewritten as
	\begin{align}\label{e-rewrittencase}
	\Omega\ci{M_1,M_2}\left[\wt{h}^{\al}(x,\la)\right]=\left[\left(\frac{\la+1}{-\al}\right)\prod_{i=1}^{r_1}n_i\prod_{i=1}^{r_2}(m_i+\al+1)\right]^{-1}\Omega\ci{M_1-1,M_2}\left[\wt{h}^{\al+1}(x,\la)\right],
	\end{align} before application.
	
	Eq.~\eqref{e-d3} will be shown using induction on the value of $T'=|t_1'|+|t_2'|$. For $T'=0$, no shift of either $M_1$ or $M_2$ is required.  As a result of $r_1=r_2=0$, all products are empty and equal to 1.  Therefore $D_3=1$ and the base case is shown.  Let $T'>0$ and assume that Eq.~\eqref{e-d3} holds whenever $|t_1'|+|t_2'|=T'-1$. Considering now $|t_1'|+|t_2'|=T'$, there are a number of cases that may occur. When both $t_1',t_2'>0$, as in the case when $T'=0$, $r_1=r_2=0$ and all products are equal to 1. Therefore, we assume that either $t_1'<0$ or $t_2'<0$. 
	Without loss of generality, consider $t_1'<0$.  The following cases can happen:
	\begin{itemize}
		\item[(i)] $t_1'<0$ and $\wt{M}_1-1$ is given by Eq.~\eqref{e-firstcase} and
		\item[(ii)] $t_1'<0$ and $\wt{M}_1-1$ is given by Eq.~\eqref{e-secondcase}.
	\end{itemize}
    The treatment of each case is similar, so we will prove only Case (ii).  In this situation, Eq.~\eqref{e-rewrittencase} holds and
	\begin{align}\label{e-somecancel}
		D_3(\al,M_1,M_2)=D_3(\al+1,M_1-1,M_2)\left(\prod_{i=1}^{r_1}n_i\prod_{i=1}^{r_2}(m_i+\al+1)\right)^{-1}. 
	\end{align}
	The induction hypothesis states
	\begin{align*}
		D_3(\al+1,M_1-1,M_2)=&\prod_{j=1}^{\wt{r}_1}\prod_{k=1}^{r_3}\left(\wt{n}_j+(\al+1)-m_k'\right)\prod_{j=1}^{r_2}\prod_{k=1}^{\wt{r}_4}\left(m_j+(\al+1)-\wt{n}_k'\right) \\
		&\quad \times\prod_{j=1}^{\wt{r}_1}\prod_{k=1}^{r_4}\left(\wt{n}_j+\wt{n}_k'+1\right)\prod_{j=1}^{r_2}\prod_{k=1}^{r_3}\left(m_j+m_k'+1\right).
	\end{align*}
	The first three double products can be rewritten as 
	\begin{align*}
		\prod_{j=1}^{\wt{r}_1}\prod_{k=1}^{r_3}\left(\wt{n}_j+(\al+1)-m_k'\right)=&\prod_{j=1}^{r_1}\prod_{k=1}^{r_3}\left(n_j+\al-m_k'\right), \\
		\prod_{j=1}^{r_2}\prod_{k=1}^{\wt{r}_4}\left(m_j+(\al+1)-\wt{n}_k'\right)=&\prod_{j=1}^{r_2}\prod_{k=1}^{r_4}\left(m_j+\al-n_k'\right)\prod_{j=1}^{r_2}\left(m_j+\al+1\right), \\
		\prod_{j=1}^{\wt{r}_1}\prod_{k=1}^{r_4}\left(\wt{n}_j+\wt{n}_k'+1\right)=&\prod_{j=1}^{r_1}\prod_{k=1}^{r_4}\left(n_j+n_k'+1\right).
	\end{align*}
	The extra products now cancel in Eq.~\eqref{e-somecancel} and the formula for $D_3$ in Eq.~\eqref{e-d3} holds. 
\end{proof}

Theorems \ref{t-new4.2} and \ref{t-new4.22} can both be simplified in the case where only the generalized Laguerre polynomial given by Eq.~\eqref{e-particularlaguerre} is considered, i.e.~there is no general solution to $\ell^{\al}$ in the last column of the Wronskian.  This is precisely Eq.~\eqref{e-particularlaguerre}. The constants generated by shifting this generalized Laguerre polynomial appear in Section \ref{s-mfunctions} but will not affect any spectral properties of the exceptional operators; they are identified here for convenience. Note that the first half of the result was shown in \cite[Theorem 4.2]{BK}.

\begin{cor}\label{c-simple4.2}
Let $f_1,\dots,f_r$ be as in Eqs.~\eqref{e-eigen1}-\eqref{e-eigen4}, $M_1$ and $M_2$ be as in Eqs.~\eqref{e-generalizedmaya1} and \eqref{e-generalizedmaya2}, $\mu$ and $\nu$ be their associated partitions after shifts $t_1$ and $t_2$ to canonical position, and $\mu'$, $\nu'$ be their associated partitions after shifts $t_1'$ and $t_2'$ to conjugate canonical position. Recall that $C_3$ and $D_3$ are defined by Eqs.~\eqref{e-c3} and \eqref{e-d3}, respectively. Then,
\begin{align*}
    \Omega\ci{M_1,M_2}^\al(x)=C_3(\al,M_1,M_2)\Omega_{\mu,\nu}^{\al-t_1-t_2}(x)=D_3(\al,M_1,M_2)\Omega_{\mu',\nu'}^{\al-t_1'-t_2'}(x).
\end{align*}
\end{cor}

\begin{proof}
The terms $C_1$ and $C_2$ from Theorem \ref{t-new4.2} are clearly produced solely by the solution $h^\al(x,\la)$ in the last column of the Wronskian when the shifts to canonical position are performed. When this last column is removed, the cumulative effect is only determined by $C_3$. Analogously, $D_3$ is the cumulative effect of shifting the Maya diagrams to conjugate canonical position. 
\end{proof}

\section{Initial Condition Normalizations}\label{s-normalizations}

Theorems \ref{t-new4.2} and \ref{t-new4.22} allow for solutions of the first and second kind to be written in terms of Wronskians indexed by partitions. This subsection will extract information from these simplified Wronskians to determine how to form a fundamental system of solutions ($u_1(x,\la)$ and $u_2(x,\la)$) for $(\cL-\la)f=0$ that is properly normalized according to Eq.~\eqref{e-initialconditions}. We leave the exact definitions of this system of solutions and the relevant quasi-derivatives to Section \ref{s-mfunctions}; now it is enough to calculate the value of these polynomial solutions at $0$. 

It is necessary to recall some notation and identities. Classical Laguerre polynomials satisfy
\begin{align*}
    L^{\al}_n(0)=\frac{(\al+1)^{(n)}}{n!}=\frac{(\al+1)(\al+2)\cdots(\al+n)}{n!},
\end{align*}
for all $n\in \mathbb \ZZ^+$. Additionally, $L^{\al}_0(0)=1$ and we set $L^{\al}_n(0)=0$ when $n \in \mathbb Z^-$. The Vandermonde determinant will be denoted as
\begin{align}\label{e-vandermonde}
    \Delta(a_1,a_2,\dots,a_r)=\prod_{i<j}(a_j-a_i).
\end{align}
Partitions will be defined via Eqs.~\eqref{e-canonicalpartition1} and \eqref{e-canonicalpartition2} and corresponding indices, e.g.~$n_i'$, used here do not correspond to the starting general Maya diagram for solutions of the first kind, but to the indices of the associated canonical position Maya diagram (or conjugate canonical position when solutions of the second kind are discussed). Note that in canonical form, $r_1=r_2=0$ so there are no $n$ or $m$-terms in the Maya diagram; and in the conjugate canonical form, $r_3=r_4=0$ so there are no $n'$ or $m'$-terms in the Maya diagram.

Solutions of the first kind, after cancelling some terms with the prefactor and using the derivative rules in Eqs.~\eqref{e-derivatives} and \eqref{e-solnderivs}, can be written as
\begin{align*}
\Omega_{\mu,\nu}\left[h^{\al}(x,\la)\right]=\begin{vmatrix}
		L^{\alpha}_{n_1}(x) &\dots &L^{\alpha}_{m_{r_2}}(-x) &M(-\la,\al+1,x)
		\\[5pt]
		(-1) L^{\alpha+1}_{n_1-1}(x) &\dots 
		&L^{\alpha+1}_{m_{r_2}}(-x) &\frac{-\la}{\al+1}M(-\la+1,\al+2,x)\\[5pt]
		\vdots &\ddots & \vdots &\vdots \\[5pt]
		(-1)^{r}L^{\alpha+r}_{n_1-r}(x) 
		&\dots 
		&L^{\alpha+r}_{m_{r_2}}(-x) &\frac{(-\la)^{(r)}}{(\al+1)^{(r)}}M(-\la+r-1,\al+r,x)
	\end{vmatrix}.
\end{align*}
Recall that $r=r(\mu)+r(\nu)$. It may occur that $n_i<r$ for some $i$, which would mean that some of the entries in the column are $0$. For this reason, let $s>0$ denote the smallest natural number such that $n_i<r$ for all $i\geq s$ and $n_i\geq r$ for all $i<s$, if it exists. If such an $s$ does not exist, we set $s=r(\mu)+1$. The sequence $\{n_i\}_{i=1}^{r(\mu)}$ is assumed to be strictly decreasing so the Wronskian has distinct entries and thus is guaranteed to have full rank. Otherwise, the generalized associated Laguerre polynomial is the constant function and we are in the setting of the classical Laguerre differential operator. Likewise, we let $s'>0$ be the smallest natural number such that $m_j<r$ for all $j\geq s'$ and $m_j\geq r$ for all $j<s'$. If such an $s'$ does not exist, set $s'=r(\nu)+1$.

The following result can be viewed as a generalization of \cite[Lemma 5.1]{D}. Recall that Assumptions \ref{a-partitions} and \ref{a-signs} hold.

\begin{theo}\label{t-first0}
Given partitions $\mu$ and $\nu$ defined via Eqs.~\eqref{e-canonicalpartition1} and \eqref{e-canonicalpartition2} and $s$, $s'$ defined above, then, $\Omega_{\mu,\nu}\left[h^{\al}(0,\la)\right]$ is equal to 
\begin{align*}
    \frac{\prod\limits_{k=1}^{r+1}(\alpha+k)^{(r+1-k)}
	\cdot
	\prod\limits_{i=1}^{s-1} (\alpha+r+1)^{(n_i-r)}
	\cdot 
	\prod\limits_{j=1}^{s'-1}(\alpha+r+1)^{(m_j-r)}
	\cdot
	\Delta(n_{\mu},m_{\nu},-\la)}
	{(\al+1)^{(r)}
	\cdot
	\prod\limits_{i=s}^{r(\mu)}(\alpha+1+n_i)^{(r-n_i)}
	\cdot 
	\prod\limits_{j=s'}^{r(\nu)}(\alpha+1+m_j)^{(r-m_j)}
	\cdot
	\prod\limits_{i=1}^{r(\mu)} n_i! 
	\cdot 
	\prod\limits_{j=1}^{r(\nu)} m_j!},
\end{align*}
where 
\begin{align*}
\Delta(n_{\mu},m_{\nu},\la)=\Delta(-n_{r(\mu)},\dots,-n_2,-n_1,\alpha+1+m_1,\alpha+1+m_2,\dots,\alpha+1+m_{r(\nu)},-\la).
\end{align*}
\end{theo}

\begin{proof}
Separate the Wronskian $\Omega_{\mu, \emptyset}\left[h^{\al}(0,\la)\right]$ into two submatrices, one with columns given by seed functions indexed by $\mu$ and the solution $h^{\al}(0,\la)$; the other with columns given by seed functions indexed by $\nu$.  Each submatrix has $r+1$ rows. Splitting the Wronskian $\Omega_{\mu, \emptyset}\left[h^{\al}(0,\la)\right]$ into these submatrices allows us to conveniently perform necessary column operations to each submatrix and then recombine the two submatrices before performing row operations. The resulting Wronskian will be a Vandermonde determinant and allow for $\Omega_{\mu,\nu}\left[h^{\al}(0,\la)\right]$ to be evaluated. 

Begin by fixing $s>0$ as above and considering the truncated Wronskian with columns given by seed functions indexed by $\mu$ and the solution $h^{\al}(0,\la)$.  This matrix has $r(\mu)$ columns and $r+1$ rows. By pulling out factors of $(-1)$, rearrange the columns so that the indices of the $n_i$ are in ascending order.  We find that this truncated Wronskian is equal to 
\begin{equation*}
			\begin{vmatrix}
			\frac{(\alpha+1)^{(n_{r(\mu)})}}{n_{r(\mu)}!} & \dots & \frac{(\alpha+1)^{(n_{s})}}{n_{s}!} & \dots & \frac{(\alpha+1)^{(n_1)}}{n_1!} & 1
			\\
			\vdots & &  &  & & 
			\\
			(-1)^{n_{r(\mu)}+1}  &  & \vdots & & &  
			\\
			0 & \ddots &  & & \vdots & \vdots  
			\\
			 & & (-1)^{s+1} & & & 
			\\
			\vdots & & 0 & \ddots &  &  
			\\
			 &  & \vdots &  &  & 
			\\
			0 & \dots & 0 &  \dots & (-1)^{r}\frac{(\alpha+r+1)^{(n_1-r)}}{(n_1-r)!} & (-1)^{r}\frac{(-\la)^{(r)}}{(\al+1)^{(r)}}
		\end{vmatrix}.
\end{equation*}
Using column operations, we can remove a factor of $1/(\al+1)^{(r)}$ from the last column, $(\al+r+1)^{(n_i-r)}/(n_i)!$ from columns where $i<s$ and $1/(\al+1+n_i)^{(r-n_i)}(n_i)!$ from columns where $i\geq s$. This leaves the Wronskian as 
\begin{equation*}
	\setlength\arraycolsep{2pt}
	\begin{vmatrix}
			(\alpha+1)^{(r)} & \dots & (\alpha+1)^{(r)} & \dots & (\alpha+1)^{(r)} & (\al+1)^{(r)}
			\\
			\vdots & & &  &  & 
			\\
			(\al+1+n_{r(\mu)})^{(r-n_{r(\mu)})}(-n_{r(\mu)})!  &  & \vdots & & & 
			\\
			0 & \ddots & & & \vdots & \vdots  
			\\
			 & & (\al+1+n_{s})^{(r-n_{s})}(-n_{s})! & & & 
			\\
			\vdots &  & 0 & \ddots&  & 
			\\
			 &  & \vdots &  &  &  
			\\
			0 & \dots & 0 & \dots & (-n_1)^{(r)} & (-\la)^{(r)}
		\end{vmatrix}.
\end{equation*}
Now using row operations, we can remove a factor of $(\al+k)^{(r+1-k)}$ from each row $k$ where $k=1,\dots,r+1$. Collecting the terms removed from this truncated Wronskian yields
\begin{equation*}
	\frac{\prod\limits_{k=1}^{r+1}(\alpha+k)^{(r+1-k)}
	\cdot
	\prod\limits_{i=1}^{s-1} (\alpha+r+1)^{(n_i-r)}}
	{(\al+1)^{(r)}
	\cdot
	\prod\limits_{i=s}^{r(\mu)}(\alpha+1+n_i)^{(r-n_i)}
	\cdot
	\prod\limits_{i=1}^{r(\mu)}n_i!}
	\setlength\arraycolsep{1pt}
	\begin{vmatrix}
		1 &  \dots & 1 & 1 \\
		\vdots &  &  &  \\
		(n_{r(\mu)})! &  &&  \\
		0 & \ddots&  \vdots & \vdots \\
		\vdots & &  &  \\
		0 & \dots & (-n_1)^{(r)} & (-\la)^{(r)}
	\end{vmatrix}.
\end{equation*}
Applying elementary row operations, the remaining determinant equals a partially simplified Vandermonde determinant with zeros in the bottom left corner. 
A standard proof of the Vandermonde determinant formula in Eq.~\eqref{e-vandermonde} begins with the last row and subtracts the previous row multiplied by $-n_{r(\mu)}$. Expanding about the first column and removing common factors yields a smaller Vandermonde matrix. The zeros in the bottom left corner of our Wronskian do not change this process and thus will not effect Eq.~\eqref{e-vandermonde} when the determinant is eventually computed. 

Now, consider the second truncated Wronskian with columns given by seed functions indexed by $\nu$ and $r+1$ rows. Set $s'>0$ as above. Fortunately, there will be no zeros in the Wronskian, but the value $s'$ will still play a role in removing terms.  Calculations the show that the evaluation of this Wronskian at $0$ is equal to 
\begin{equation*}
	\begin{vmatrix}
		\frac{(\alpha+1)^{(m_1)}}{m_1!} & \frac{(\alpha+1)^{(m_2)}}{m_2!} & \dots & \frac{(\alpha+1)^{(m_{r(\nu)})}}{m_{r(\nu)}!} \\
		\frac{(\alpha+2)^{(m_1)}}{m_1!} & \frac{(\alpha+2)^{(m_2)}}{m_2!} & \dots & \frac{(\alpha+2)^{(m_{r(\nu)})}}{m_{r(\nu)}!} \\
		\vdots & \vdots & \ddots & \vdots \\
		\frac{(\alpha+r+1)^{(m_1)}}{m_1!} & \frac{(\alpha+r+1)^{(m_2)}}{m_2!} & \dots & \frac{(\alpha+r+1)^{(m_{r(\nu)})}}{(m_{r(\nu)})!} 
    \end{vmatrix}.
\end{equation*}
Using column operations, we can remove a factor of $(\alpha+r+1)^{(m_j-r)}/m_j!$ from all columns where $j<s'$ and a factor of $1/(\alpha+1+m_j)^{(r-m_j)}(m_j)!$ from remaining columns. The Wronskian is then
\begin{equation*}
	\begin{vmatrix}
		(\alpha+1)^{(r)}\cdot 1 & \dots & (\alpha+1)^{(r)}\cdot 1 \\
		(\alpha+2)^{(r-1)}(\alpha+1+m_1) & \dots & (\alpha+2)^{(r-1)}(\alpha+1+m_{r(\nu)}) \\
		\vdots & \ddots & \vdots \\
		1\cdot(\alpha+1+m_1)^{(r)}& \dots & 1\cdot(\alpha+1+m_{r(\nu)})^{(r)}
	\end{vmatrix}.
\end{equation*}
Now remove a factor of $(\al+k)^{(r+1-k)}$ from each row $k=1,\dots,r+1$. Altogether, the second truncated Wronskian equals
\begin{equation*}
	\frac{\prod\limits_{k=1}^{r+1}(\al+k)^{(r+1-k)}
	\prod\limits_{j=1}^{s'-1}(\alpha+r+1)^{(m_j-r)}}
	{\prod\limits_{j=s'}^{r(\nu)}(\alpha+1+m_j)^{(r-m_j)}
	\prod\limits_{j=1}^{r(\nu)}m_j!}
	\setlength\arraycolsep{-.5pt}
	\begin{vmatrix}
		1 & \dots & 1 \\
		(\alpha+1+m_1) & \dots & (\alpha+1+m_{r(\nu)}) \\
		\vdots & \ddots & \vdots \\
		(\alpha+1+m_1)^{(r(\nu)-1)} & \dots & (\alpha+1+m_{r(\nu)})^{(r(\nu)-1)}
	\end{vmatrix}.
\end{equation*}
Applying the same row operations as in the analysis of the first truncated Wronskian, this determinant is simply a Vandermonde determinant with additional rows. 

The solution of the first kind $\Omega_{\mu,\nu}\left[h^{\al}(0,\la)\right]$ merges the two truncated matrices analyzed above, with some columns rearranged.  The result is a power of $(-1)$ as an additional factor. The same column operations can be performed as above, and the row operations for the two truncated matrices agree. The result follows by collecting the above terms and by Eq.~\eqref{e-vandermonde}.
\end{proof}

The solution of the second kind can similarly be evaluated at $0$ using the methods of Theorem \ref{t-first0}. We define $\bs$ to be the analog of $s$ for the sequence $\{n_i'\}_{i=1}^{r(\mu')}$ and $\bs'$ the analog for the sequence $\{m_j'\}_{j=1}^{r(\nu')}$.

\begin{cor}\label{c-second0}
Given partitions $\mu'$ and $\nu'$ defined via Eqs.~\eqref{e-canonicalpartition1} and \eqref{e-canonicalpartition2} and $\bs$, $\bs'$ defined above, then, $\Omega_{\mu',\nu'}\left[\wt{h}^{\al}(0,\la)\right]$ is equal to 
\begin{align*}
    \frac{\prod\limits_{i=1}^{r+1}(-\alpha+k)^{(r+1-k)}
	\cdot
	\prod\limits_{j=1}^{\bs'-1} (-\alpha+r+1)^{(m_j'-r)}
	\cdot 
	\prod\limits_{i=1}^{\bs-1}(-\alpha+r+1)^{(n_i'-r)}
	\cdot
	\Delta(m_{\nu'}',n_{\mu'}',-\la-\al)}
	{(1-\al)^{(r)}
	\cdot
	\prod\limits_{j=\bs'}^{r(\nu')}(-\alpha+1+m_j')^{(r-m_j')}
	\cdot 
	\prod\limits_{i=\bs}^{r(\mu')}(-\alpha+1+n_i')^{(r-n_i')}
	\cdot
	\prod\limits_{j=1}^{r(\nu')} m_j'! 
	\cdot
	\prod\limits_{i=1}^{r(\mu')} n_i'!},
\end{align*}
where 
\begin{align*}
\Delta(m_{\nu'}',n_{\mu'}',-\la-\al)=\Delta(-m_{r(\nu)'}',&\dots,-m_2',-m_1', \\
&-\alpha+1+n_1',-\alpha+1+n_2',\dots,-\alpha+1+n_{r(\mu')}',-\la-\al).
\end{align*}
\end{cor}

\begin{proof}
The solution of the second kind $\Omega_{\mu',\nu'}\left[\wt{h}^{\al}(x,\la)\right]$ by definition has a prefactor of $x^{\al(r'+1)}$. Removing a factor of $x^{-\al}$ from each column cancels with this prefactor and yields a Wronskian with only functions of the types given in Eq.~\eqref{e-eigen1}, indexed by $\{m_j'\}_{j=1}^{r(\nu')}$, and Eq.~\eqref{e-eigen2}, indexed by $\{n_i'\}_{i=1}^{r(\mu')}$, with the function $M(-\la-\al,1-\al,x)$ as the input for the last column. The partition indices have thus been changed to $(\nu',\mu')$. It is then possible to apply Theorem \ref{t-first0} to evaluate this Wronskian. Note the parameter for the seed functions is now $-\al$, as opposed to $\al$ from Theorem \ref{t-first0}. The result follows. 
\end{proof}

The proof of Theorem \ref{t-first0} can be modified to fit the simplified case where the generalized Laguerre polynomial is considered. This is essentially \cite[Lemma 5.1]{D} in the notation of this manuscript. Here, $s$ and its analogs must be slightly altered as there are $r-1$ instead of $r$ derivatives taken. Let $s>0$ be the smallest natural number such that $n_i<r-1$ for all $i\geq s$ and $n_i\geq r-1$ for all $i<s$. If such an $s$ does not exist, set $s=r(\mu)+1$. Define $s'$, $\bs$ and $\bs'$ analogously.

\begin{cor}\label{c-simple0}
Given partitions $\mu$ and $\nu$ defined via Eqs.~\eqref{e-canonicalpartition1} and \eqref{e-canonicalpartition2} and $s$, $s'$ defined above, then
\begin{align*}
\Omega_{\mu,\nu}^{\al}(0)=
    \frac{\prod\limits_{k=1}^{r}(\alpha+k)^{(r-k)}
	\cdot
	\prod\limits_{i=1}^{s-1} (\alpha+r)^{(n_i-r+1)}
	\cdot 
	\prod\limits_{j=1}^{s'-1}(\alpha+r)^{(m_j-r+1)}
	\cdot
	\Delta(n_{\mu},m_{\nu})}
	{\prod\limits_{i=s}^{r(\mu)}(\alpha+1+n_i)^{(r-1-n_i)}
	\cdot 
	\prod\limits_{j=s'}^{r(\nu)}(\alpha+1+m_j)^{(r-1-m_j)}
	\cdot
	\prod\limits_{i=1}^{r(\mu)} n_i! 
	\cdot 
	\prod\limits_{j=1}^{r(\nu)} m_j!}.
\end{align*}
Furthermore, given $\mu'$ and $\nu'$ and $\bs$, $\bs'$ defined above,
\begin{align*}
    \Omega_{\mu',\nu'}^\al(0)=\frac{\prod\limits_{i=1}^{r}(-\alpha+k)^{(r-k)}
	\cdot
	\prod\limits_{j=1}^{\bs'-1} (-\alpha+r)^{(m_j'-r+1)}
	\cdot 
	\prod\limits_{i=1}^{\bs-1}(-\alpha+r)^{(n_i'-r+1)}
	\cdot
	\Delta(m_{\nu'}',n_{\mu'}')}
	{\prod\limits_{j=\bs'}^{r(\nu')}(-\alpha+1+m_j')^{(r-1-m_j')}
	\cdot 
	\prod\limits_{i=\bs}^{r(\mu')}(-\alpha+1+n_i')^{(r-1-n_i')}
	\cdot
	\prod\limits_{j=1}^{r(\nu')} m_j'! 
	\cdot
	\prod\limits_{i=1}^{r(\mu')} n_i'!}.
\end{align*}
\end{cor}

\begin{proof}
It is necessary to slightly alter the column operations of Theorem \ref{t-first0} in order to account for the fact that the Wronskian only has $r$ rows and columns. Hence, there are $r-1$ derivatives taken and this is reflected in the altered rising factorials and terms in the numerators of the expressions. The results follow after these alterations. 
\end{proof}

\section{Weyl $m$-Functions}\label{s-mfunctions}

Sections \ref{s-manipulations} and \ref{s-normalizations} have provided the tools necessary to obtain spectral information about self-adjoint extensions of exceptional Laguerre expressions given by two general Maya diagrams. 
 
 Let $\cL$ be a general Laguerre-type differential expression acting on $L^2[(0,\infty),\cW(x)dx]$ with maximal domain $\cD\ti{max}$. Given fixed $M_1$ and $M_2$, the partitions $\mu$, $\nu$ are determined by Eqs.~\eqref{e-canonicalpartition1} and \eqref{e-canonicalpartition2} and the values $r(\mu)$, $r(\nu)$ are determined by Eqs.~\eqref{e-rs1} and \eqref{e-rs2}. The weight function possesses a factor of $e^{-x}x^{\bal}$, where $\bal=\al'+r(\mu)+r(\nu)=\al'+r$ is fixed due to Lemma \ref{l-ortho}. It is assumed that $\bal>-1$ so $\cL$ is in the limit-circle case at $x=0$. The value $\al'=\al-t_1-t_2$ is determined by the shifts $t_1$ and $t_2$ of the partitions from Theorem \ref{t-new4.2}. The weight function $\cW(x)$ depends on a generalized Laguerre polynomial $\Omega\ci{M_1,M_2}^\al(x)$, which we identify with $\phi$ from Section \ref{s-prelims}. 
 
\begin{remme}\label{r-alphas}
The attentive reader may question why $\al'$ is used in the definition of $\bal$ for the order of the underlying Hilbert space instead of $\al-t_1'-t_2'$, which appears in Theorem \ref{t-new4.22}. Is the order $\al-t_1'-t_2'$ even useful? Indeed, Theorem \ref{t-new4.2} shows that multiplication of the Wronskian by a certain factor can change the order of the parameter $\al$ within the Wronskian, causing some ambiguity. The solution of the second kind in Theorem \ref{t-new4.22} could be shifted to canonical position, also yielding a parameter $\al-t_1-t_2$. However, it is not possible to perform the evaluation in Corollary \ref{c-second0} with this setup. Hence, the order $\al-t_1'-t_2'$ is still necessary but not related to the underlying Hilbert space.
\end{remme}
 
 The Frobenius analysis of Section \ref{s-mayaforxop} for general Laguerre XOP expressions indicates that there are two linearly independent solutions at the regular singular point $x=0$: $y_1(x)=1$ and $y_2(x)=x^{-\bal}$. These two functions are in the maximal domain $\cD\ti{max}$ as they both belong to $L^2[(0,\infty),\cW(x)]$ and are solutions to the equation $(\cL-\la)f=0$. For the sake of normalization, we will use the functions
\begin{align}\label{e-phi}
\wt{y}_1(x):=-\frac{\Omega\ci{M_1,M_2}^\al(0)}{\bal} \quad\text{ and }\quad \wt{y}_2(x):=\Omega\ci{M_1,M_2}^\al(0)\cdot x^{-\bal}
\end{align}
instead of $y_1(x)=1$ and $y_2(x)=x^{-\bal}$ as particular solutions. Note that $\Omega\ci{M_1,M_2}^\al(x)$ is a polynomial with real coefficients so $\Omega\ci{M_1,M_2}^\al(0)$ exists, and it was already assumed that $\Omega\ci{M_1,M_2}^\al(0)\neq0$. Hence, $\wt{y}_1$ and $\wt{y}_2$ are well-defined and non-zero but slightly different from the formulations found in Section \ref{s-type1}.

The sesquilinear form associated to the expression $\cL$ is defined via Eq.~\eqref{e-xopsesqui}. Let $f,g\in \cD\ti{max}$ and define
\begin{align}\label{e-boundaryops}
    \Gamma_0f=f^{[0]}(0)&:=[f,\wt{y}_2]\ci\cL(0)=\lim_{x\to0^+}-\frac{\bal f(x)+xf'(x)}{\Omega\ci{M_1,M_2}^\al(0)} ,\\
    \Gamma_1f=f^{[1]}(0)&:=[f,\wt{y}_1]\ci\cL(0)=\lim_{x\to0^+}\frac{x^{\bal+1}}{\bal\cdot \Omega\ci{M_1,M_2}^\al(0)}f'(x).
\end{align}
The choice of the functions $\wt{y}_1$ and $\wt{y}_2$ immediately imply that $[\wt{y}_1,\wt{y}_2]\ci\cL(0)=1$. The functions $\wt{y}_1$ and $\wt{y}_2$ adhere to Definition \ref{d-quasideriv}, so Proposition \ref{p-btsetup} implies that $\{\CC,\Gamma_0,\Gamma_1\}$ is a boundary triple for $\cD\ti{max}$. Explicitly, it is easy to calculate that
\begin{align*}
    \langle\Gamma_1 f,\Gamma_0 g\rangle-\langle\Gamma_0 f,\Gamma_1 g\rangle=[f,g]\ci\cL(0).
\end{align*}
As in Eq.~\eqref{e-generalsetup}, this boundary triple naturally defines two self-adjoint extensions of the exceptional Laguerre expression $\cL$: $\cL_\infty$ and $\cL_0$. 

The solutions of the first and second kind can be modified to form a fundamental system for the equation $(\cL\ti{max}-\la)f=0$. Define
\begin{equation}\label{e-fundsystem1}
\begin{aligned}
u_1(x,\la)&:=-\frac{\Omega\ci{M_1,M_2}^\al(0) \cdot \Omega\ci{M_1,M_2}\left[h^{\al}(x,\la)\right]}
{\bal \cdot \Omega\ci{M_1,M_2}\left[h^{\al}(0,\la)\right]}
=
-\frac{\Omega\ci{M_1,M_2}^\al(0) \cdot \Omega_{\mu,\nu}\left[h^{\al'}(x,\la+t_1)\right]}
{\bal \cdot \Omega_{\mu,\nu}\left[h^{\al'}(0,\la+t_1)\right]}, 
\end{aligned}
\end{equation}
and
\begin{equation}\label{e-fundsystem2}
\begin{aligned}
u_2(x,\la)&:=-\frac{\Omega\ci{M_1,M_2}^\al(0)  \cdot x^{-\bal} \cdot \Omega\ci{M_1,M_2}\left[\wt{h}^{\al}(x,\la)\right]}
{\Omega\ci{M_1,M_2}\left[\wt{h}^{\al}(0,\la)\right]}\\
&=
-\frac{\Omega\ci{M_1,M_2}^\al(0)  \cdot x^{-\bal} \cdot \Omega_{\mu',\nu'}\left[\wt{h}^{\al-t_1'-t_2'}(x,\la+t_1')\right]}
{\Omega_{\mu',\nu'}\left[\wt{h}^{\al-t_1'-t_2'}(0,\la+t_1')\right]},
\end{aligned}
\end{equation}
where each equality follows from shifting the Maya Diagrams in both the numerator and denominator, with the constants from Theorems \ref{t-new4.2} and \ref{t-new4.22} cancelling out. Notice that if the prefactor of $\Omega\ci{M_1,M_2}\left[\wt{h}^{\al}(x,\la)\right]$ is ignored, then the result is guaranteed to be in $L^2[(0,\infty),\cW(x)dx]$ by the intertwining property of the Darboux transform. The multiplication by $x^{-\bal}$ above ensures this remains the case with the standard prefactor included. 

\begin{ass}\label{a-type3}
We assume that the $x^{-\bal}$ term is multiplied to the solution of the second kind and therefore defines the behavior of $u_2(x,\la)$ near $x=0$.
\end{ass}

\begin{remme}\label{r-type3}
Multiplication by the $x^{-\bal}$ term discussed in Assumption \ref{a-type3} may instead be applied to the solution of the first kind, but this would change our notation significantly. In particular, this alternative convention can be translated into what we use here by changing the Maya diagrams and treating the function of the second kind as the principal solution, which is a polynomial by the notation of Section \ref{s-mayaforxop}. The alternate convention has been used to state examples like the Type III Laguerre XOPs described in \cite{LLMS}, where the eigenvalues generated by the polynomials can all be of the form $n-\al$ for $n\in\NN_0$. 
\end{remme}

\begin{cor}\label{c-ics}
The fundamental system $(u_1(x,\la),u_2(x,\la))$ for the equation $(\cL\ti{max}-\la)f=0$, given in Eqs.~\eqref{e-fundsystem1} and \eqref{e-fundsystem2}, satisfies the initial conditions
\begin{align*}
    \left( \begin{array}{cc}
u_1^{[0]}(0,\la) &  u_2^{[0]}(0,\la) \\
u_1^{[1]}(0,\la) & u_2^{[1]}(0,\la)
\end{array} \right)
=
    \left( \begin{array}{cc}
1 & 0 \\
0 & 1
\end{array} \right).
\end{align*}
\end{cor}

\begin{proof}
Recall that both $\Omega\ci{M_1,M_2}\left[h^{\al}(x,\la)\right]$ and $\Omega\ci{M_1,M_2}\left[\wt{h}^{\al}(x,\la)\right]$ are formulated as polynomials. The identities are easily verified by plugging $u_1(x,\la)$ and $u_2(x,\la))$ into the quasi-derivatives. 
\end{proof}

Let $\chi(x,\la)$ be a general deficiency element satisfying $(\ell-\la)\chi(x,\la)=0$. Then, $\chi$ can be written as a linear combination of the solutions which compose the fundamental system. In the case of the classical Laguerre operator, $\chi$ is the Tricomi confluent hypergeometric function which can be expressed via \cite[Equation (13.2.42)]{DLMF} as
\begin{align}\label{e-oldtricomi}
U(-\la,\al+1,x)=\frac{\Gamma(-\al)}{\Gamma(-\la-\al)}M(-\la,\al+1,x)+\frac{\Gamma(\al)}{\Gamma(-\la)}z^{-\al}M(-\la-\al,1-\al,x).
\end{align}
The intertwining property of the Darboux transform then allows us to define a general deficiency element of the exceptional differential expression. 

\begin{prop}\label{p-defelement}
A general deficiency element $\chi\ci\cL(x,\la)$ of $(\cL-\la)\chi(x,\la)=0$ can be written as 
\begin{align*}
    \chi\ci\cL(x,\la)=\fC u_1(x,\la)+\fD u_2(x,\la),
\end{align*}
where 
\begin{equation}\label{e-finalconstants}
\begin{aligned}
    \fC&=-\frac{\bal
    \cdot 
    C
    \cdot
    \Gamma(-\al)
    \cdot 
    \Omega_{\mu,\nu}\left[h^{\al'}(0,\la+t_1)\right]}
    {\Gamma(-\la-\al)
    \cdot 
    \Omega\ci{M_1,M_2}^\al(0) }, \\
    \fD&=-\frac{D
    \cdot 
    \Gamma(\al)
    \cdot 
    \Omega_{\mu',\nu'}\left[\wt{h}^{\al-t_1'-t_2'}(0,\la+t_1')\right]}
    {\Gamma(-\la)
    \cdot
    \Omega\ci{M_1,M_2}^\al(0)}.
\end{aligned}
\end{equation}
In particular, we have that $\chi\ci\cL^{[0]}(x,\la)=\fC$ and $\chi\ci\cL^{[1]}(x,\la)=\fD$.
\end{prop}

\begin{proof}
Let $A\ci{M_1,M_2}$ denote the multi-step Darboux transform with seed functions indexed by $M_1$ and $M_2$. The Wronskian is linear within its entries because differentiation and multiplication are. Applying $A\ci{M_1,M_2}$ to Eq.~\eqref{e-oldtricomi} then yields
\begin{align*}
    \chi\ci\cL(x,\la)&=A\ci{M_1,M_2}\left[\chi(x,\la)\right]\\&=A\ci{M_1,M_2}\left[\frac{\Gamma(-\al)}{\Gamma(-\la-\al)}h^{\al}(x,\la)+\frac{\Gamma(\al)}{\Gamma(-\la)}\wt{h}^{\al}(x,\la)\right] \\
    &=\frac{\Gamma(-\al)}{\Gamma(-\la-\al)}\cdot\Omega\ci{M_1,M_2}\left[h^{\al}(x,\la)\right]+\frac{\Gamma(\al)}{\Gamma(-\la)}x^{-\bal}\cdot\Omega\ci{M_1,M_2}\left[\wt{h}^{\al}(x,\la)\right].
\end{align*}
Recall the constants stemming from Theorems \ref{t-new4.2} and \ref{t-new4.22}, and stated in Eqs.~\eqref{e-C} and \eqref{e-D}, are denoted by $C$ and $D$. Then,
\begin{align*}
    \chi\ci\cL(x,\la)&=C\frac{\Gamma(-\al)}{\Gamma(-\la-\al)}\cdot\Omega_{\mu,\nu}\left[h^{\al'}(x,\la+t_1)\right]+D\frac{\Gamma(\al)}{\Gamma(-\la)}x^{-\bal}\cdot\Omega_{\mu',\nu'}\left[\wt{h}^{\al-t_1'-t_2'}(x,\la+t_1')\right].
\end{align*}
The formulation of $\chi\ci\cL(x,\la)$ from the statement of the proposition follows. The initial conditions of Corollary \ref{c-ics} yield the values of the quasi-derivatives of this deficiency element and the result follows. 
\end{proof}

Let $\tau\in\RR\cup\{\infty\}$ and $\cL_\tau$ refer to the self-adjoint operator acting via Eq.~\eqref{e-ltau}. The values of the quasi-derivatives of a general deficiency element allows for the definition of a Weyl $m$-function for every self-adjoint extension $\cL_\tau$.

\begin{cor}\label{c-mfunctions}
For $\la\in\rho(\cL_\infty)$, 
\begin{equation}\label{e-minfty}
\begin{aligned}
    M_\infty(\la)&=\frac{\chi\ci\cL^{[1]}(0,\la)}{\chi\ci\cL^{[0]}(0,\la)}=\frac{\fD}{\fC} \\
    &=\frac{D
    \cdot
    \Gamma(-\la-\al) 
    \cdot 
    \Gamma(\al)
    \cdot 
    \Omega_{\mu',\nu'}\left[\wt{h}^{\al-t_1'-t_2'}(0,\la+t_1')\right]}
    {\bal
    \cdot
    C
    \cdot 
    \Gamma(-\la)
    \cdot 
    \Gamma(-\al)
    \cdot
    \Omega_{\mu,\nu}\left[h^{\al'}(0,\la+t_1)\right]}.
\end{aligned}
\end{equation}
For $\la\in\rho(\cL_\infty)\cup\rho(\cL_0)$,
\begin{equation}\label{e-m0}
\begin{aligned}
M_0(\la)&=-\frac{\chi\ci\cL^{[0]}(0,\la)}{\chi\ci\cL^{[1]}(0,\la)}=-\frac{\fC}{\fD} \\
&=-\frac{\bal
    \cdot
    C
    \cdot 
    \Gamma(-\la)
    \cdot 
    \Gamma(-\al)
    \cdot
    \Omega_{\mu,\nu}\left[h^{\al'}(0,\la+t_1)\right]}
    {D
    \cdot
    \Gamma(-\la-\al) 
    \cdot 
    \Gamma(\al)
    \cdot 
    \Omega_{\mu',\nu'}\left[\wt{h}^{\al-t_1'-t_2'}(0,\la+t_1')\right]}.
\end{aligned}
\end{equation}
For $\la\in\rho(\cL_\infty)\cup\rho(\cL_\tau)$, 
\begin{equation}\label{e-mtau}
\begin{aligned}
&M_\tau(\la)=\frac{1+\tau M_\infty(\la)}{\tau-M_\infty(\la)} \\
&=\frac{\bal
    C
    \Gamma(-\la)
    \Gamma(-\al) \cdot
    \Omega_{\mu,\nu}\left[h^{\al'}(0)\right]
+
\tau D
    \Gamma(-\la-\al) 
    \Gamma(\al) \cdot
    \Omega_{\mu',\nu'}\left[\wt{h}^{\al-t_1'-t_2'}(0)\right]
}
{\tau \bal
    C 
    \Gamma(-\la)
    \Gamma(-\al) \cdot
    \Omega_{\mu,\nu}\left[h^{\al'}(0)\right]
    -
    D
    \Gamma(-\la-\al) 
    \Gamma(\al) \cdot
    \Omega_{\mu',\nu'}\left[\wt{h}^{\al-t_1'-t_2'}(0)\right]},
\end{aligned}
\end{equation}
where the dependence of the solutions on $\la$ is suppressed for ease of notation.
\end{cor}

\begin{proof}
An application of Proposition \ref{p-btsetup} to the boundary triple $\{\RR,\Gamma_0,\Gamma_1\}$ and the deficiency element $\chi\ci\cL(x,\la)$ yields Eq.~\eqref{e-minfty}. Likewise, the identity $M_0(\la)=-1\backslash M_\infty(\la)$ yields Eq.~\eqref{e-m0}. The $m$-function of $\cL_\tau$ is obtained from Eq.~\eqref{e-generalbt}.
\end{proof}

The spectrum of the self-adjoint operator $\cL_\infty$ consists of the points that are poles of $M_\infty$. The function $\Gamma(-\la-\al)$ is responsible for the eigenvalues $\la=n-\al$, for $n\in\NN_0$. These are exactly the eigenvalues of the extension $\cL_\infty$ for the classical Laguerre operator. In that case, both partitions are empty so $C$, $D$ and the two evaluations at $0$ are all equal to $1$. 
 
The self-adjoint extension $\cL_0$ will possess the exceptional solution of the first kind and therefore have the family of exceptional polynomial eigenfunctions in its domain by Lemma \ref{l-ortho}. Eigenvalues of $\cL_0$ are simply the poles of $M_0$. The corresponding extension of the classical Laguerre expression contains the Laguerre polynomials and has eigenvalues $\la=n$ for $n\in\NN_0$. The function $\Gamma(-\la)$ is responsible for these eigenvalues in Eq.~\eqref{e-m0}.

Comparing these expressions to the classical Laguerre operator, we conclude that eigenvalues $\la$ not stemming from $\Gamma(-\la)$ or $\Gamma(-\la-\al)$ are manufactured by the Darboux transform. The evaluations $\Omega_{\mu,\nu}\left[h^{\al'}(0,\la+t_1)\right]$ and $\Omega_{\mu',\nu'}\left[\wt{h}^{\al-t_1'-t_2'}(0,\la+t_1')\right]$ do not depend on $\la$ except for the Vandermonde determinants $\Delta(n_\mu,m_\nu,-\la)$ and $\Delta(m_{\nu'}',n_{\mu'}',-\la-\al)$, respectively in the numerators. The values $C$ and $D$ can also have factors of $\la$ in their numerators or denominators depending on whether the shifts are positive or negative.

\begin{remme}\label{r-friedrichskvn1}
The Friedrichs extension of $\cL$ is identified as the self-adjoint extension whose boundary condition is given by $[f,u]\ci\cL(0)=0$ for a principal solution $u$ by \cite{MZ}. The two linearly independent particular solutions (near 0) of $\cL$ used in our boundary triple were $\wt{y}_1$ and $\wt{y}_2$. In practice, which solution is principal thus depends on the parameter $\bal$. Hence, one of $\cL_0$ and $\cL_\infty$ will be the Friedrichs extension and, if $\cL>0$, the other will be the Krein--von Neumann extension, see e.g.~\cite[Equation (5.4.26)]{BdS}, but identifying which is which is not always straightforward. 

The easiest way to distinguish these important extensions is to observe that the Friedrichs extension will have the greatest first eigenvalue because it has the greatest lower bound among all extensions, see e.g.~\cite[Proposition 5.3.6]{BdS}. Any dependence of the lowest eigenvalues on $\bal$ may then result in the Friedrichs extension switching from $\cL_0$ to $\cL_\infty$, or vice versa.
\end{remme}

\begin{remme}\label{r-type1check}
Section \ref{s-type1} analyzes the Type I Laguerre exceptional operator $\cL^{I,\al}_m$ and finds that $\sigma(\cL_0)=n$ and $\sigma(\cL_\infty)=(-m-\al)\cup(n+1-\al)$ for $n\in\NN_0$ and fixed $m\in\NN$. These eigenvalues can also be recovered from Eqs.~\eqref{e-minfty} and \eqref{e-m0}. Note it is possible to match the $m$--functions from these two methods but we focus only on the spectra here for the sake of simplicity.

The Darboux transform in this case is identified by \cite[Section 5.2]{BK} to consist of seed functions indexed by $\mu=\emptyset$, $\nu=(m)$ and starting parameter $\al-1$. Hence, $t_1=t_2=t_1'=0$, $t_2'=-m-1$ and $\bal=(\al-1)+1=\al$. Theorem \ref{t-new4.22} sets $C_1=C_2=0$ and Theorem \ref{t-first0} finds that $\Omega_{\mu,\nu}[h^{\al-1}(0,\la)]$ has a factor of $-\la-(\al-1)-1-m=-\la-\al-m$ in the numerator stemming from the Vandermonde determinant. The only dependence on $\la$ in $\fC$ is therefore a factor of $(-\la-\al-m)\Gamma(-\la)$.

After shifting to conjugate canonical position, $M_2=(m-1,\dots,0)$. Theorem \ref{t-new4.22} sets $D_1=0$ and $D_2$ has a factor of $\prod_{k'}(-\la-(\al-1)-1-k')$ with $k'=0,\dots,m-1$ in the denominator. Corollary \ref{c-second0} also generates a factor, for $k=1,\dots m$,
\begin{align*}
\prod_{k}(-\la-[(\al-1)-(-m-1)]+k')=\prod_{k}(-\la-\al-m+k),
\end{align*}
in the numerator of $\Omega_{\mu',\nu'}[\wt{h}^{\al-1}(0,\la)]$. These contributions cancel and the only dependence on $\la$ in $\fD$ is therefore a factor of $\Gamma(-\la-\al+1)$.

Generating the expressions for $M_0$ and $M_\infty$ via Eqs.~\eqref{e-minfty} and \eqref{e-m0} allows for easy identification of the eigenvalues, which agree with the conclusions of Section \ref{s-type1}. 
\end{remme}

\begin{remme}\label{r-pert}
Perturbation theory says that the spectrum of $\cL_0$ will smoothly shift to the spectrum of $\cL_\infty$ as the parameter $\tau$ increases and eigenvalues of all self-adjoint extensions will be simple, see e.g.~\cite{BFL,S2,S}.
\end{remme}

\section{Example}\label{s-example}

	To illustrate the results of Theorem \ref{t-new4.2}, we present the following example.  Consider Maya Diagrams $M_1=(\emptyset\big| 3,2)$ and $M_2=(1,0\big|\emptyset)$, illustrated in Figure \ref{f-original}.
		
	\begin{figure}[h!]\caption{Original Maya Diagrams} \label{f-original}
		\begin{tikzpicture}[scale=.8]
			\draw (0,0) grid(8,1);
			\draw [fill] (.5,0.5) circle [radius=0.1];
			\draw [fill] (1.5,0.5) circle [radius=0.1];
			\draw [fill] (2.5,0.5) circle [radius=0.1];
			\draw [fill] (5.5,0.5) circle [radius=0.1];
			\draw [fill] (6.5,0.5) circle [radius=0.1];
			\draw [thick] (3,-.3)--(3,1.3);
			\node at (2.5,1.2) {$-1$};
			\node at (3.5,1.2) {$0$};
			\node at (4,-0.75) {$M_1$};
		\end{tikzpicture}\hfill
		\begin{tikzpicture}[scale=.8]
			\draw (0,0) grid(8,1);
			\draw [fill] (.5,0.5) circle [radius=0.1];
			\draw [fill] (1.5,0.5) circle [radius=0.1];
			\draw [fill] (2.5,0.5) circle [radius=0.1];
			\draw  [thick] (5,-.3)--(5,1.3);
			\node at (4.5,1.2) {$-1$};
			\node at (5.5,1.2) {$0$};
			\node at (4,-0.75) {$M_2$};
		\end{tikzpicture}
	\end{figure}
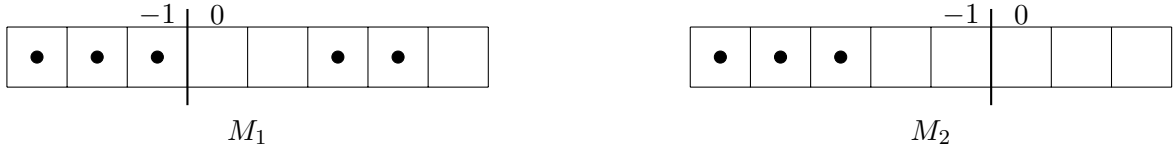

	From $M_1$ and $M_2$, it follows that $n=(3,2)$, $m=\emptyset$, $n'=\emptyset$, and $m'=(1,0)$ as well as $r_1=r_3=2$ and $r_2=r_4=0$.  Observe that $M_1$ is already in canonical form; thus $t_1=0$.  To put $M_2$ into canonical form, a shift of $t_2=2$ is required.  Thus $t_1(M_1)=(\emptyset \big| 3,2)$ and $t_2(M_2)=(\emptyset\big|\emptyset)$.  The visual representation for the canonical form of $M_1$ and $M_2$ may be seen in Figure \ref{f-canonical}, where all boxes to the left of the pictured subdiagram are filled and those to the right are empty. 
	
	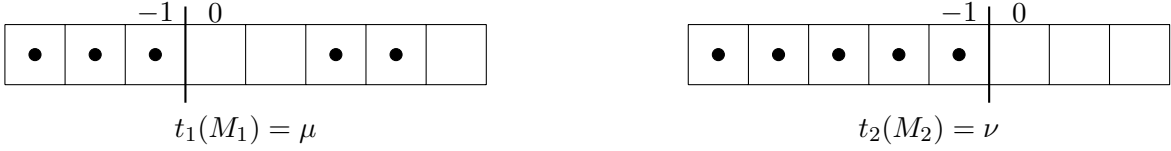
\begin{figure}[h!]\caption{Canonical Maya Diagrams} \label{f-canonical}
		\begin{tikzpicture}[scale=.8]
			\draw (0,0) grid(8,1);
			\draw [fill] (.5,0.5) circle [radius=0.1];
			\draw [fill] (1.5,0.5) circle [radius=0.1];
			\draw [fill] (2.5,0.5) circle [radius=0.1];
			\draw [fill] (5.5,0.5) circle [radius=0.1];
			\draw [fill] (6.5,0.5) circle [radius=0.1];
			\draw [thick] (3,-.3)--(3,1.3);
			\node at (2.5,1.2) {$-1$};
			\node at (3.5,1.2) {$0$};
			\node at (4,-0.75) {$t_1(M_1)=\mu$};
		\end{tikzpicture}\hfill
		\begin{tikzpicture}[scale=.8]
			\draw (0,0) grid(8,1);
			\draw [fill] (.5,0.5) circle [radius=0.1];
			\draw [fill] (1.5,0.5) circle [radius=0.1];
			\draw [fill] (2.5,0.5) circle [radius=0.1];
			\draw [fill] (3.5,0.5) circle [radius=0.1];
			\draw [fill] (4.5,0.5) circle [radius=0.1];
			\draw  [thick] (5,-.3)--(5,1.3);
			\node at (4.5,1.2) {$-1$};
			\node at (5.5,1.2) {$0$};
			\node at (4,-0.75) {$t_2(M_2)=\nu$};
		\end{tikzpicture}
	\end{figure}
	Using Eqs.~\eqref{e-eigen1} and \eqref{e-eigen2} in Eq.~\eqref{e-firstgen} yields
	\begin{align}\label{e-omega0}
		\Omega\ci{M_1,M_2}\left[h^\alpha(x,\lambda)\right]&=x^{2(\alpha+3)} \Wr\left[L_3^\alpha(x), L_2^\alpha(x), x^{-\alpha}L_1^{-\alpha} (x), x^{-\alpha}L_0^{-\alpha}(x),M(-\lambda,\alpha+1,x)\right].
	\end{align}
	Following the procedure outlined in the proof of Theorem \ref{t-new4.2}, Eq.~\eqref{e-r30red} of Corollary \ref{c-0reduction} will be applied twice to shift $M_1$ and $M_2$ into $t_1(M_1)=\mu$ and $t_2(M_2)=\nu$. To demonstrate how Theorem \ref{t-new4.2} works, we illustrate the first application of the shift on $M_2$.  

    Beginning with Eq.~\eqref{e-omega0}, we first factor out $x^{-\al}$ from each column
    \begin{align*}
        \Omega\ci{M_1,M_2}\left[h^\alpha(x,\lambda)\right]&=x^{2(\al+3)}x^{-5\al}\Wr\left[x^\al L_3^\al(x), x^\al L_2^\al (x), L_1^{-\al}(x), 1, x^\al M(-\la, \al+1,x)\right]. 
    \end{align*} To compute the Wronskian, we expand along the first row, fourth column entry, using appropriate derivative rules
    \begin{align*}
        \Omega\ci{M_1,M_2}\left[h^\alpha(x,\lambda)\right]&=x^{2(\al+3)}x^{-5\al}\Wr\left[(3+\al)x^{\al-1}L_3^{\al-1}(x), (2+\al)x^{\al-1} L_2^{\al-1} (x), \right.\\
        &\hspace{1.5 in} \left.-L_0^{-\al+1}(x), \al x^{\al-1} M(-\la, \al,x)\right]. 
    \end{align*} Factoring out constants and $x^{\al-1}$ from each factor in the Wronskian yields
    \begin{align*}
        \Omega\ci{M_1,M_2}\left[h^\alpha(x,\lambda)\right]&=-\al (\al+2)(\al+3)x^{\al+2}\\
        & \quad \quad \quad \quad \times\Wr\left[L_3^{\al-1}(x), L_2^{\al-1} (x), x^{-\al+1}L_0^{-\al+1}(x), M(-\la, \al,x)\right]\\
        &=-\al(\al+2)(\al+3)\Omega_{\wt{M}_1,\wt{M}_2}\left[h^{\al-1}(x,\lambda)\right],
    \end{align*} where $\wt{M}_1=M_1=(\emptyset\big|3,2)$ and $\wt{M}_2=M_2+1=(0\big|\emptyset)$. Observe that just as Eq.~\eqref{e-r30red} states, the parameter $\al$ has been shifted down by one while $\la$ does not change.  One additional application of Eq.~\eqref{e-r30red} gives 
	\begin{align*}
		\Omega\ci{M_1, M_2} &\left[h^\alpha(x,\lambda)\right]
		\\&=-\alpha(\alpha+2)(\alpha+3)x^{-3\alpha+6}\cdot\Wr \left[x^{\alpha-1}L_3^{\alpha-1}(x),x^{\alpha-1}L_2^{\alpha-1}(x),1, x^{\alpha-1}M(-\lambda, \alpha, x)\right]\\
		&=-(\alpha-1)\alpha(\alpha+1)(\alpha+2)^2(\alpha+3)\cdot\Wr \left[L_3^{\alpha-2}(x),L_2^{\alpha-2}(x),1, x^{\alpha-1}M(-\lambda, \alpha-1, x)\right]\\
		&=C_1\cdot C_2\cdot C_3\cdot\Omega_{\mu, \nu}^{\alpha-2}(x, \lambda),
	\end{align*} where $C_1=1$, $C_2=\alpha(\alpha-1)$, and $C_3=(\alpha+1)(\alpha+2)^2(\alpha+3)$.  Keep in mind that the signs of coefficients $C_1$, $C_2$, and $C_3$ are suppressed.  This result is consistent with that guaranteed by Theorem \ref{t-new4.2}.

    Now, we aim to adjust $M_1$ and $M_2$ into conjugate canonical form, see Figure \ref{f-conjugate} below.  Observe that $M_1$ must be shifted left by four units; thus $t'_1=-4$.  To put $M_2$ into conjugate canonical form, a shift of $t'_2=2$ is required.  Note that in this example, $M_2$ has the same canonical and conjugate canonical form.  Thus $t'_1(M_1)=(3,2 \big|\emptyset)$ and $t'_2(M_2)=(\emptyset\big|\emptyset)$. Following the procedure outlined in Theorem \ref{t-new4.22}, Eq.~\eqref{e-r30red2} of Corollary \ref{c-0reduction2} is applied twice to shift $M_2$ into conjugate canonical form, and the inverse of Eq.~\eqref{e-r40red2} is applied twice to $M_1$, followed by two applications of Eq.~\eqref{e-r40red2} to shift $M_1$ into conjugate canonical form.  Note that there is some freedom as to how the shifts are applied.  For example, $M_1$ may be put into canonical form before $M_2$; however, the order of the shifts applied to $M_1$ must be done in the prescribed order.  
    
	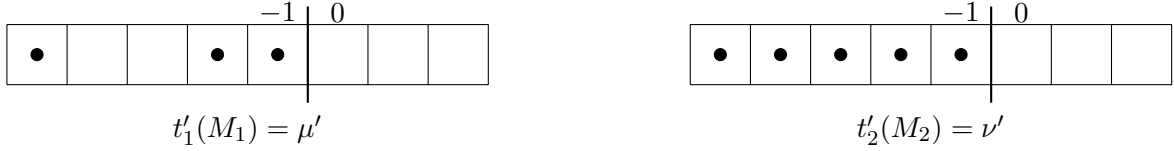
\begin{figure}[h!]\caption{Conjugate Canonical Maya Diagrams} \label{f-conjugate}
		\begin{tikzpicture}[scale=.8]
			\draw (0,0) grid(8,1);
			\draw [fill] (.5,0.5) circle [radius=0.1];
			\draw [fill] (3.5,0.5) circle [radius=0.1];
			\draw [fill] (4.5,0.5) circle [radius=0.1];
			\draw [thick] (5,-.3)--(5,1.3);
			\node at (4.5,1.2) {$-1$};
			\node at (5.5,1.2) {$0$};
			\node at (4,-0.75) {$t'_1(M_1)=\mu'$};
		\end{tikzpicture}\hfill
		\begin{tikzpicture}[scale=.8]
			\draw (0,0) grid(8,1);
			\draw [fill] (.5,0.5) circle [radius=0.1];
			\draw [fill] (1.5,0.5) circle [radius=0.1];
			\draw [fill] (2.5,0.5) circle [radius=0.1];
			\draw [fill] (3.5,0.5) circle [radius=0.1];
			\draw [fill] (4.5,0.5) circle [radius=0.1];
			\draw  [thick] (5,-.3)--(5,1.3);
			\node at (4.5,1.2) {$-1$};
			\node at (5.5,1.2) {$0$};
			\node at (4,-0.75) {$t'_2(M_2)=\nu'$};
		\end{tikzpicture}
	\end{figure}
    Consequently, Eq.~\eqref{e-secondgen} may be written, up to a change of sign, as
    \begin{align*}
        \Omega\ci{M_1,M_2}\left[\wt{h}^\al(x,\lambda)\right]&=D_1\cdot D_2\cdot D_3 \cdot \Omega_{\mu',\nu'}\left[\wt{h}^{\al+2}(x,\lambda-4)\right]\,,
    \end{align*}
    where 
    \[
        D_1=\frac{\al(\al+1)(\al+2)(\al+3)}{\lambda(\lambda-1)},\; 
        D_2=\frac{(\lambda+\alpha)(\lambda+\alpha+1)}{(\al+2)(\al+3)}, \mbox{ and }
        D_3=(\al+1)(\al+2)^2(\al+3).
    \]
    
    Now, we turn our attention to calculating $\Omega\ci{M_1,M_2}\left[h^\al(0,\lambda)\right]$ and $\Omega\ci{M_1,M_2}\left[\wt{h}^\al(0,\lambda)\right]$, which requires the use of Theorem \ref{t-first0} and Corollary \ref{c-second0} respectively.  For the shift to canonical position, $\mu=(2,2)$ and $\nu=\emptyset$; $r(\mu)=2$ and $r(\nu)=0$; and $s=3$ and $s'=2$. Note that $\mu$ is even, satisfying Assumption \ref{a-partitions}.  Applying Theorem \ref{t-first0} gives
    \begin{align*}
        \Omega_{\mu, \nu}\left[h^{\al-2}(0,\lambda)\right]=\frac{\alpha(\alpha+1)(3-\lambda)(2-\lambda)}{12}\,.
    \end{align*}
    For the shift to conjugate canonical position, $\mu'=(2,2)$ and $\nu'=\emptyset$; $r(\mu')=2$ and $r(\nu')=0$; and $\bs=3$ and $\bs'=2$.  Applying Corollary \ref{c-second0} gives
    \begin{align*}
        \Omega_{\mu',\nu'}\left[\wt{h}^{\al+2}(0,\lambda-4)\right]=\frac{(\al-1)\al(\lambda-1)\lambda}{12}\,.
    \end{align*}
    
    The value $\bal=\al-0-2+2+0=\al$ satisfies the requirement that $\bal>-1$ because $\al>-1$. Hence, the exceptional Laguerre differential expression $\cL$ is limit-circle at $x=0$ and limit-point at $x=\infty$.  Now, it is possible to explicitly write out $\fC$ and $\fD$ as
    \begin{equation*}
        \fC=\frac{(\al-1)\al^3(\al+1)^2(\al+2)^2(\al+3)(3-\lambda)(2-\lambda)\Gamma(-\al)}{12\Gamma(-\lambda-\alpha)\left[\Omega\ci{M_1,M_2}^\al(0)\right]^2}
    \end{equation*}
    and
    \begin{equation*}
        \fD=\frac{(\al-1)\al^2(\al+1)^2(\al+2)^2(\al+3)(\lambda+\al)(\lambda+\al+1)\Gamma(\al)}{12\Gamma(-\lambda)\left[\Omega\ci{M_1,M_2}^\al(0)\right]^2}.
    \end{equation*}
    The quotient of $\fD$ and $\fC$ yields
    \begin{align}\label{e-exMInf}
        M_\infty(\lambda)=\frac{(\lambda+\al)(\lambda+\al+1)\Gamma(\al)\Gamma(-\lambda-\al)}{\al(3-\lambda)(2-\lambda)\Gamma(-\al)\Gamma(-\lambda)}.
    \end{align} To find the spectrum of the self-adjoint operator $L_\infty$, consider where Eq.~\eqref{e-exMInf} has poles. The poles occur for eigenvalues $\lambda=n-\alpha$ for $n\in\mathbb N_0$ as well as $\lambda=2$ and $\lambda=3$.  However, the pole at $\la=-\al$ is cancelled by the factor of $\la+\al$ in the numerator. Thus $\sigma(L_\infty)=\{n-\alpha\}_{n\in \mathbb N}\cup\{2,3\}$.
    
    Considering the reciprocal of $M_\infty$, we get
    \begin{align}
        M_0(\lambda)=\frac{\al(3-\lambda)(2-\lambda)\Gamma(-\al)\Gamma(-\lambda)}{(\lambda+\al)(\lambda+\al+1)\Gamma(\al)\Gamma(-\lambda-\al)}
    \end{align} which has poles at $\lambda=n$ for all $n \in \mathbb N_0$ as well as $\lambda=-\al$ and $\lambda=-\alpha-1$. The poles at $\la=2,3$ are cancelled by factors in the numerator. Thus $\sigma(L_0)=\{n\}_{n \in \NN_0/\{2,3\}} \cup \{-\al,-\al-1\}$.

\begin{remme}\label{r-friedrichsvkn2}
The removal of the eigenvalues $\la=2,3$ from $\sigma(\cL_0)$ and $\la=-\al$ from $\sigma(\cL_\infty)$ is crucial, otherwise $\sigma(\cL_0)\cap\sigma(\cL_\infty)$ is nontrivial. 
Assume some $\la_0\in\RR$ belongs to $\sigma(\cL_0)\cap\sigma(\cL_\infty)$. Then the spectral measures of the two extensions are not mutually singular, violating the Aronszajn--Donoghue theorem for rank-one perturbations, see e.g.~\cite{A,S2}

Furthermore, Remark \ref{r-friedrichskvn1} identifies that $\cL_\infty$ is the Friedrichs extension for all values of $\al>-1$, as comparing the first eigenvalues shows $1-\al>-1-\al$.
\end{remme}

\section{Conclusions}\label{s-conclusion}

Exceptional Laguerre-type symmetric expressions $\cL$, which have XOPs as eigenfunctions for a self-adjoint extension, may be derived from the classical Laguerre symmetric operator with parameter $\al$ by applying Darboux transforms. These are written as Wronskians of seed functions that meet Assumption \ref{a-partitions} and the seed functions are described via the Maya Diagrams $M_1$ and $M_2$ or, in special situations, as partitions $\mu$, $\nu$, $\mu'$ or $\nu'$.

The expression $\cL$ acts on functions in the Hilbert space $L^2[(0,\infty), \cW(x)]$ that satisfy standard differentiability criteria. The weight function 
\begin{align*}
    \cW(x)=\frac{x^{\bal}e^{-x}}{\left[\Omega^{\al'}_{\mu,\nu}(x)\right]^2} \quad\text{ for } x>0,
\end{align*}
where $\al'=\al-t_1-t_2$, $\bal=\al'+r$, and $\Omega^{\al}_{\mu,\nu}(x)$ is defined by Eq.~\eqref{e-particularlaguerre}. The expression $\cL$ is limit-circle at $x=0$ and limit-point at $x=\infty$ when $\al$ is chosen so that $-1<\bal<1$. 

Laguerre-type XOPs of order $n\in\NN_{\mu,\nu}$ are written as $L^{\bal}_{\mu,\nu,n}$ and the sequence forms a complete orthogonal set in $L^2[(0,\infty), \cW]$ by Lemma \ref{l-ortho}, see Remark \ref{r-poly} for more details. These polynomials all belong to a family of eigenfunctions given by $\Omega\ci{M_1,M_2}[h^\al(x,\la)]$ and are in a self-adjoint extension of $\cL$ denoted $\cL_0$. 

Eigenvalues of self-adjoint extensions of $\cL$, denoted $\cL_\tau$ in general, can be extracted from an explicit Weyl $m$-function given by Eq.~\eqref{e-mtau} generated by a boundary triple. The resulting $m$-function can be compared to the $m$-function of self-adjoint extensions of the classical Laguerre symmetric operator. Differences stem from two important processes: shifting Maya diagrams into canonical/conjugate canonical position, and evaluating Wronskians with these shifted diagrams at $x=0$. Other spectral information, such as the strength of the point masses in the spectral measure, can also be obtained from the $m$-function, see e.g.~\cite{BFL}.

The examples presented in Sections \ref{s-type1} and \ref{s-example} provide insight into how eigenvalues are changed by different types of Maya diagrams. Inverse spectral theory, where a set of eigenvalues is given and a corresponding set of seed functions within the Wronskian is determined (thereby fixing the differential equation and self-adjoint extension), however, remains unknown in general and warrants further investigation. 


\section*{Acknowledgements}

The authors would like to thank Niels Bonneux for correspondence that provided a detailed proof of \cite[Lemma 5.1]{D}. This correspondence was instrumental in the proofs of Theorem \ref{t-first0} and Corollary \ref{c-second0}.



\begin{thebibliography}{xx}            


\bibitem{A} N.~Aronszajn, {\em On a problem of Weyl in the theory of singular Sturm--Liouville equations}, Am.~J.~Math.~{\bf 79} (1957), 597--610.

\bibitem{ALS} M.~Atia, L.~Littlejohn, J.~Stewart, {\em The spectral theory of the $X_1$-Laguerre polynomials}, Adv.~Dyn.~Syst.~Appl.~{\bf 8} (2013) 81-92.

\bibitem{BdS} J.~Behrndt, S.~Hassi, H.~de Snoo, {\em Boundary value problems, Weyl functions, and differential operators}, Monographs in Mathematics, Birkh\"auser {\bf 108} (2020).

\bibitem{B} S.~Bochner, {\em \"Uber Sturm–Liouvillesche polynomsysteme}, Math.~Z.~{\bf 29} (1929), 730--736.

\bibitem{BFL} M.~Bush, D.~Frymark, C.~Liaw, {\em Singular boundary conditions for Sturm--Liouville operators via perturbation theory}, to appear in Canad.~J.~Math. (2022), 1--37.

\bibitem{BK} N.~Bonneux, A.~Kuijlaars, {\em Exceptional Laguerre polynomials}, Stud.~Appl.~Math.~{\bf 141} (2018), 547--595.

\bibitem{C} M.~Crum, {\em Associated Sturm--Liouville systems}, Q.~J.~Math.~Oxford Ser.~{\bf 6} (1955), 121--127.

\bibitem{D} A.~Dur\'an, {\em Exceptional Meixner and Laguerre orthogonal polynomials}, J.~Approx.~Theory {\bf 184} (2014), 176--208.

\bibitem{DP} A.~Dur\'an, M.~P\'erez, {\em Admissibility condition for exceptional Laguerre polynomials}, J.~Math. Anal.~Appl.~{\bf 424} (2015), 1042--1053.

\bibitem{F} D.~Frymark, {\em Boundary triples and Weyl $m$-functions for powers of the Jacobi differential operator}, J.~Differential Equations {\bf 269} (2020), 7931--7974.

\bibitem{GGM} M.~Garc\'{i}a--Ferrero, D.~G\'{o}mez--Ullate, R.~Milson, {\em A Bochner type characterization theorem for exceptional orthogonal polynomials}, J.~Math. Anal.~Appl.~{\bf 472} (2019), 584--626.

\bibitem{GGM2} D.~G\'omez--Ullate, Y.~Grandati, R.~Milson, {\em Shape invariance and equivalence relations for pseudo-Wronskians of Laguerre and Jacobi polynomials}, J.~Phys.~A.: Math.~Theor.~{\bf 51} (2018), 345201.

\bibitem{GGM3} \bysame, {\em Spectral theory of exceptional Hermite polynomials}, From Operator Theory to Orthogonal Polynomials, Combinatorics, and Number Theory, Oper.~Theory~Adv.~Appl., Birkh\"auser Verlag {\bf 285} (2021), 173--196.

\bibitem{GKM1} D.~G\'omez-Ullate, N.~Kamran, R.~Milson, {\em An extended class of orthogonal polynomials
defined by a Sturm–Liouville problem}, J.~Math.~Anal.~Appl.~{\bf 359} (2009), 352–-367.

\bibitem{GKM2} \bysame, {\em An extension of Bochner’s problem: exceptional
invariant subspaces}, J.~Approx.~Theory {\bf 162} (2010), 987--1006.

\bibitem{GKM3} \bysame, {\em Exceptional orthogonal polynomials and the Darboux transformation}, J.~Phys.~A.~: Math.~Theor.~{\bf 43} (2010), 434016.

\bibitem{GKM4} \bysame, {\em Two-step Darboux transformations and exceptional Laguerre polynomials}, J.~Math.~Anal.~Appl.~{\bf 387} (2012), 410--418.

\bibitem{GKM5} \bysame, {\em A conjecture on exceptional orthogonal polynomials}, Found.~Comput.~Math.~{\bf 13} (2013), 615-666.

\bibitem{GLN} F.~Gesztesy, L.~Littlejohn, R.~Nichols, {\em On self-adjoint boundary conditions for singular Sturm--Liouville operators bounded from below}, J.~Differential Equations {\bf 269} (2020), 6448--6491.

\bibitem{G} Y.~Grandati, {\em Multistep DBT and regular rational extensions of the isotonic oscillator}, Ann.~Phys.~{\bf 327} (2012), 2411--2431.

\bibitem{KLO} J.~Stewart Kelly, C.~Liaw, J.~Osborn, {\em Moment representations of exceptional X1 orthogonal polynomials}, J.~Math. Anal.~Appl.~{\bf 455} (2017), 1848--1869.

\bibitem{LLK} C.~Liaw, L.~Littlejohn, J.~Stewart Kelly, {\em Spectral analysis for the exceptional $X_m$-Jacobi Equation}, Electron.~J.~Diff.~Equ.~{\bf 194} (2015), 1-10.

\bibitem{LLMS} C.~Liaw, L.~Littlejohn, R.~Milson, J.~Stewart, {\em The spectral analysis of three families of exceptional Laguerre polynomials}, J.~Approx.~Theory {\bf 202} (2016), 5--41.

\bibitem{LLSW} C.~Liaw, L.~Littlejohn, J.~Stewart, Q.~Wicks, {\em A spectral study of the second-order exceptional X1-Jacobi differential expression and a related non-classical Jacobi differential expression}, J.~Math. Anal.~Appl.~{\bf 422} (2015), 212--239.

\bibitem{MQ} M.~Marquette, C.~Quesne, {\em New families of superintegrable systems from Hermite and Laguerre exceptional orthogonal polynomials}, J.~Phys.~A.~: Math.~Theor.~{\bf 54} (2013), 042102.

\bibitem{MR} B.~Midya, B.~Roy, {\em Exceptional orthogonal polynomials and exactly solvable potential sin position dependent Schr\"{o}dinger Hamiltonians}, Phys.~Lett.~A.~{\bf 45} (2009), 4117-.

\bibitem{MZ} M.~Marletta, A.~Zettl, {\em The Friedrichs extension of singular differential operators}, J.~Differential Equations {\bf 160} (2000), 404--421.

\bibitem{N} M.~Naimark, {\em Linear Differential Operators Part I, II}, Frederick Ungar Publishing Co., New York, NY (1972).

\bibitem{DLMF} {\em NIST Digital Library of Mathematical Functions}, F.~Olver, A.~Daalhuis, D.~Lozier, B.~Schneider, R.~Boisvert, C.~Clark, B.~Miller, and B.~Saunders, eds.~release 1.0.22 of 2019-03-15, \url{http://dlmf.nist.gov/}.

\bibitem{OS}S.~Odake, R.~Sasaki, {\em Infinitely many shape invariant potentials and new orthogonal polynomials} Phys.~Lett.B {\bf 679} (2009) 414--417.

\bibitem{OS2} \bysame, {\em Another set of infinitely many ($X_\ell$) Laguerre polynomials} Phys.~Lett.~B {\bf 684} (2010) 173--176.

\bibitem{PT}  G.~Post, A.~Turbiner, {\em Classification of linear differential operators with an invariant subspace of monomials}, Russ.~J.~Math.~Phys.~{\bf 3} (1995), 113--122.

\bibitem{Q} C.~Quesne, {\em Exceptional orthogonal polynomials, exactly solvable potentials, and super symmetry}, Pys.~Lett.~A {\bf 41} (2008), 392001.

\bibitem{STZ} R.~Sasaki, S.~Tsujimoto, A.~Zhedanov, {\em Exceptional Laguerre and Jacobi polynomials and the corresponding potentials through Darboux-Crum transformations}, J.~Phys.~A.: Math.~Theor.~{\bf 43} (2010), 315204.

\bibitem{S2} B.~Simon, {\em Spectral analysis of rank one perturbations and applications}, Mathematical quantum theory II. Schr\"odinger operators (Vancouver, BC, 1993), CRM Proc.~Lecture Notes, {\bf 8}, Amer.~Math.~Soc., Providence, RI.~ (1995), 109--149.

\bibitem{S} \bysame, {\em Trace Ideals and Their Applications}, 2nd ed.~American Mathematical Society, Providence, RI (2005).

\end{thebibliography}
\end{document}